\pdfoutput=1
\RequirePackage{ifpdf}
\ifpdf 
\documentclass[pdftex]{sigma}
\else
\documentclass{sigma}
\fi

\usepackage{tikz-cd}
\usepackage{mathrsfs}

\numberwithin{equation}{section}
\newtheorem{Theorem}{Theorem}[section]
\newtheorem*{Theorem*}{Theorem}
\newtheorem{Lemma}[Theorem]{Lemma}
\newtheorem{Proposition}[Theorem]{Proposition}
\newtheorem{Conjecture}[Theorem]{Conjecture}
\newtheorem*{Conjecture*}{Conjecture}
{\theoremstyle{definition}
\newtheorem{Definition}[Theorem]{Definition}
\newtheorem{Example}[Theorem]{Example}
\newtheorem{Remark}[Theorem]{Remark}
}

\usetikzlibrary{matrix}
\usetikzlibrary{arrows}

\def\loc{\operatorname{loc}}

\def\g{\operatorname{g}}

\def\dom{\operatorname{dom}}
\def\range{\operatorname{range}}

\def\id{\operatorname{id}}

\def\proj{\operatorname{proj}}

\def\Diff{\operatorname{Diff}}

\def\ev{\operatorname{ev}}
\def\Aut{\operatorname{Aut}}
\def\Fl{\operatorname{Fl}}

\def\FF{{\mathcal{F}}}
\def\BB{{\mathcal{B}}}
\def\CC{{\mathcal{C}}}

\def\WW{{\mathcal{W}}}
\def\AA{{\mathcal{A}}}
\def\CC{{\mathcal{C}}}

\def\UU{{\mathcal{U}}}
\def\VV{{\mathcal{V}}}

\def\HH{{\mathcal{H}}}
\def\GG{{\mathcal{G}}}
\def\PP{{\mathcal{P}}}
\def\NN{{\mathcal{N}}}
\def\YY{{\mathcal{Y}}}
\def\XX{{\mathcal{X}}}
\def\OO{{\mathcal{O}}}

\def\DS{{\mathscr{D}}}

\def\SS{{\mathscr{S}}}

\def\RB{{\mathbb{R}}}

\def\NB{{\mathbb{N}}}

\def\XF{{\mathfrak{X}}}

\def\pf{{\mathfrak{p}}}

\def\vpf{{\mathfrak{vp}}}

\def\tsep#1{%
 {\@tempdima=\ht\strutbox\advance\@tempdima #1
 \vrule height \@tempdima depth 0pt width 0pt\nobreak\hspace{0pt}%
 }%
}
\def\bsep#1{%
 {\@tempdima=\dp\strutbox \advance\@tempdima #1
 \nobreak\hspace{0pt}\vrule height 0pt depth \@tempdima width 0pt
 }%
}

\begin{document}

\newcommand{\arXivNumber}{2006.14271}

\renewcommand{\PaperNumber}{043}

\FirstPageHeading

\ShortArticleName{The Holonomy Groupoids of Singularly Foliated Bundles}

\ArticleName{The Holonomy Groupoids\\ of Singularly Foliated Bundles}

\Author{Lachlan Ewen MACDONALD}

\AuthorNameForHeading{L.E.~MacDonald}

\Address{School of Mathematical Sciences, The University of Adelaide,\\ Adelaide, South Australia, 5000, Australia}
\Email{\href{mailto:lachlan.macdonald@adelaide.edu.au}{lachlan.macdonald@adelaide.edu.au}}

\ArticleDates{Received December 07, 2020, in final form April 20, 2021; Published online April 28, 2021}

\Abstract{We define a notion of connection in a fibre bundle that is compatible with a~singular foliation of the base. Fibre bundles equipped with such connections are in plentiful supply, arising naturally for any Lie groupoid-equivariant bundle, and simultaneously generalising regularly foliated bundles in the sense of Kamber--Tondeur and singular foliations. We~define hierarchies of diffeological holonomy groupoids associated to such bundles, which arise from the parallel transport of jet/germinal conservation laws. We~show that the groupoids associated in this manner to trivial singularly foliated bundles are quotients of~Androulidakis--Skandalis holonomy groupoids, which coincide with Androulidakis--Skandalis holonomy groupoids in the regular case. Finally we prove functoriality of all our constructions under appropriate morphisms.}


\Classification{53C05; 53C12; 53C29}\vspace{1mm}

\section{Introduction}

In this paper, we extend the notion of a partial connection in a fibre bundle to the singular setting, obtaining \textit{singular partial connections}. Fibre bundles with singular partial connections, which we refer to as \textit{singularly foliated bundles}, are induced by Lie groupoid actions on~fib\-re bundles, and generalise singular foliations and regularly foliated bundles. We~use certain diffeological pseudo-bundles consisting of jets/germs of sections that are conserved by the foliation's flow to systematically construct hierarchies of holonomy groupoids for singularly foliated bundles as diffeological quotients of path spaces, and show that our constructions are functorial under suitably defined morphisms of singularly foliated bundles. For a trivial bundle over a~singular foliation, the associated holonomy groupoid is a quotient of the groupoid defined by Androulidakis--Skandalis~\cite{iakovos2} by a relation which identifies groupoid elements if they act in the same way on the germs of first integrals, and in particular coincides with the Winkelnkemper--Phillips holonomy groupoid of a regular foliation.

Singular foliations are involutive, locally finitely generated families of vector fields on mani\-folds. As~famously proved by Stefan~\cite{stefan} and Sussmann~\cite{sussmann}, such objects integrate to give decompositions of their ambient manifolds into immersed submanifolds, possibly of differing dimension, called leaves. Singular foliations are ubiquitous in mathematics and its applications. For instance, every Poisson manifold has a singular foliation by symplectic leaves, and conversely a singular foliation of a manifold by symplectic leaves suffices to determine a Poisson structure~\cite{cf1}. More generally, any integrable Dirac manifold admits a singular foliation by presymplectic leaves~\cite{courant}. Singular foliations also generalise regular foliations, which are among the primary instances of Connes' noncommutative geometries~\cite{folops, ncg}.

An essential construction for the noncommutative perspective is the \textit{holonomy groupoid} of~a~regular foliation, which was introduced by Winkelnkemper~\cite{wink} as a model for the leaf
space. As~described by Phillips~\cite{holimper}, the holonomy groupoid is in a precise sense the smallest desingularisation of the naive ``space of leaves'' obtained as the quotient by leaves that admits a~(locally Hausdorff) manifold structure. It is upon the locally Hausdorff holonomy groupoid of a foliation (or \'{e}tale versions thereof) that a great deal of progress has been made in index theory~\cite{ben6, ben3, ben4, ben7, ben2, ben5, pcr, cs, goro3, heitsch2, mac2, macr1, ms} and equivariant/cyclic cohomology~\cite{cyctrans,hopf1,diffcyc, backindgeom,crainic1,goro1,goro2,mac1,mosc1, moscorangi}. An alternative toolbox for the study of regular foliations that has been developing since the nineteen-nineties is \textit{diffeology}~\cite{hector, HMVSC}, which provides a way of doing differential topology on conventionally badly behaved spaces $\XX$ by declaring which maps from Euclidean domains into $\XX$ are smooth. Recent progress by the author in this area~\cite{mac3} shows that the holonomy groupoid of a regular foliation is just the largest of an infinite family of \textit{diffeological} holonomy groupoids constructed using solutions of parallel transport differential equations in diffeological bundles. Thus, while the Winkelnkemper holonomy groupoid is the smallest Lie groupoid that integrates a regular foliation, it is far from being the smallest diffeological groupoid that does so.

Defining holonomy groupoids for singular foliations dates back to the mid nineteen-eighties with work of Pradines and Bigonnet~\cite{pradines,pb}. Significant further progress was made by Debord in~\cite{debord, debord2} in the study of holonomy groupoids for singular foliations arising from Lie algebroids whose anchor maps are injective on a dense set (these types of foliations are now known as \textit{Debord foliations}~\cite[Definition~3.6]{lgls}). Such foliations are special in that their holonomy groupoids are Lie groupoids. At present, the most general family of singular foliations for which holonomy groupoids can be defined are those associated to locally finitely generated, involutive families of vector fields, in the spirit of those originally studied by Stefan and Sussmann. The~holonomy groupoids of such general foliations were formulated by Androulidakis and Skandalis in~\cite{iakovos2}. Holonomy groupoids at this level of generality are topologically pathological, but, as is evident in the recent preprint~\cite{VilGar1} of Garmendia and Villatoro, are \textit{diffeologically} quite well-behaved, arising as spaces of classes of leafwise paths identified via their maps on transversal slices. The~years since the Androulidakis--Skandalis construction have seen a great deal of further research conducted into singular foliations and their holonomy, see for instance~\cite{iakovos7,iakovos3,iakovos,iakovos4,ori1,GarZam,VilGar1}.

The present paper constitutes a generalisation of the holonomy groupoid constructions in~\cite{mac3} to singular foliations, and is inspired in part by the recent work of Garmendia and Villatoro~\cite{VilGar1}, who showed how to recover the Androulidakis--Skandalis holonomy groupoid as a quotient of a diffeological path space. In~the author's view, the primary contribution of this paper is a~novel perspective on the holonomy of singular foliations which arises from parallel transport of conservation laws. In~particular, this places the holonomy of singular foliations in the same realm of differential geometry that deals with symmetries and conservation laws of differential equations in the sense of~\cite{varbi,olver}. In~addition, the diffeological pseudo-bundles of germs that we~introduce in this paper are shown to be extensions of jet bundles, which are closely related to~(but topologically distinct from) \'{e}tale spaces of sheaves. We~believe that these pseudo-bundles may be of independent interest and utility. Let us now outline the content of the paper in more detail.

Section~\ref{sec2} consists of a recollection of the well-known definitions and results from singular foliations, jet bundle theory and diffeology that will be required for our constructions later in~the paper. We~remark here that our notation $\Gamma_{k}$ for the $k^{\rm th}$ order jet bundles differs from the~$J^{k}$ that is usually seen in the literature~-- this is to ensure compatibility with the pseudo-bundles of~germs that we introduce in the following section.

Section~\ref{sec3} is where we introduce the key diffeological constructions with which the holonomy groupoids of singularly foliated bundles can be systematically constructed. In~particular, we~asso\-ciate to any sheaf $\SS$ of smooth sections of a fibre bundle $\pi_{B}\colon B\rightarrow M$ over a manifold $M$ a~cano\-nical diffeological pseudo-bundle $\Gamma_{\g}(\SS)$ over $M$, whose fibre over $x\in M$ consists of all the germs $[\sigma]^{\g}_{x}$ at $x$ of elements $\sigma$ of $\SS$ defined around $x$. When $\SS$ is the sheaf of all sections of~$\pi_{B}$, the ``pseudo-bundle of germs'' $\Gamma_{\g}(\pi_{B})$ ought to be thought of as a ``completion'' of the usual tower of jet bundles $\Gamma_{k}(\pi_{B})$ associated to $\pi_{B}$, which is sufficiently rich to capture the behaviour of \textit{non-analytic} smooth sections. The~concept of jet prolongation of a vector field to a jet bundle is extended to \textit{germinal} prolongation of a vector field to a bundle of germs, which is a crucial component in the definition of our holonomy groupoids.

We include in Section~\ref{sec3} a discussion of the relationship between pseudo-bundles of germs and classical sheaf theoretic concepts. In~particular, we show in Proposition~\ref{relsheaf1} that any sui\-t\-ably smooth morphism of sheaves of sections induces a morphism of the corresponding pseudo-bundles of germs, following which we give a counterexample to the converse being true. Finally, Remark~\ref{relsheaf4} shows that while the pseudo-bundle of germs of a sheaf is isomorphic \textit{as a set} to the well-known \'{e}tale space associated to the sheaf, the topology it inherits from its diffeology is (often strictly) contained in the usual \'{e}tale topology. These considerations regarding the topology of pseudo-bundles of germs are not required anywhere in our constructions, and are included out of independent interest.

Section~\ref{sec3} concludes by recalling the diffeological path categories $\PP(\XX)$ of diffeological spa\-ces~$\XX$ introduced in~\cite{mac3}, and generalises the leafwise path category of a regular foliation introduced therein to the singular case. Elements of the leafwise path category $\PP(\FF)$ of a singular foliation~$\FF$ are triples $\big(\gamma,[X]^{\g},d\big)$, where $X$ is some locally-defined, time-dependent vector field in $\FF$, $d>0$ a~real number, and $\gamma\colon \RB_{\geq0} = [0,\infty)\rightarrow M$ an integral curve of $X$ such that $X$ vanishes in a~neigh\-bour\-hood of $\gamma(0)$ and of $\gamma([d,\infty))$. This definition draws from the analogous definition used by Garmendia and Villatoro in~\cite{VilGar1}. We~also define an abstract notion of holonomy groupoid associated to a lifting map into paths of a pseudo-bundle, which serves as the framework for the constructions of Section~\ref{sec4}.

In Section~\ref{sec4} we generalise the notion of a singular foliation to a \textit{singularly foliated bundle}. In~the same way that regularly foliated bundles in the sense of Kamber and Tondeur~\cite{folbund} are defined in terms of a partial connection on the total space of the bundle, our singularly foliated bundles are defined in terms of what we call \textit{singular partial connections}. Roughly speaking, a~singular partial connection $\ell$ in a fibre bundle over a singular foliation allows us to lift vector fields from the foliation of the base to fields on the total space. We~show that singularly foliated bundles generalise singular foliations and regularly foliated bundles, and are associated to all bundles that are equivariant under the action of a Lie groupoid. Associated to any singularly foliated bundle $\pi_{B}\colon B\rightarrow M$ with foliation $\FF$ of the base are pseudo-bundles $\pi_{B}^{k,\FF}\colon \Gamma_{k}(\pi_{B})^{\FF}\rightarrow M$ of germs (if~$k=\g$) or jets (if $k\leq\infty$) of sections which are invariant, to $k^{\rm th}$ order, under flows of~$\FF$. These $\FF$-invariant jets/germs generalise the jets/germs of ``distinguished sections'' considered in~\cite{mac3}, and may be thought of as conservation laws for $\FF$. We~prove the following.

\begin{Theorem}[see Theorem~\ref{liftsmooth}]
Let $\pi_{B}\colon B\rightarrow M$ be a singularly foliated bundle, and let $k$ denote any of the symbols $0,\dots,\infty,\g$. If $\pi_{B}$ admits local sections over each point which are invariant to $k^{\rm th}$ order, then there is a smooth lifting map $L\big(\pi_{B}^{k,\FF}\big)\colon \PP(\FF)\times_{s,\pi_{B}^{k,\FF}}\Gamma_{k}(\pi_{B})^{\FF}\rightarrow \PP\big(\Gamma_{k}(\pi_{B})^{\FF}\big)$ sending leafwise paths in $M$ to the solutions of a parallel transport differential equation in $\Gamma_{k}(\pi_{B})^{\FF}$.
\end{Theorem}

In particular, invariant jets/germs about a point can be extended to invariant jets/germs along any path in $\FF$ starting at that point. Each of the lifting maps $L\big(\pi_{B}^{k,\FF}\big)$ of a singularly foliated bundle $\pi_{B}\colon B\rightarrow M$ induces a transport functor from the leafwise path category $\PP(\FF)$ to the diffeological groupoid of diffeomorphisms between the fibres of $\Gamma_{k}(\pi_{B})^{\FF}$. The~fibres of this functor determine an equivalence relation on $\PP(\FF)$, and the quotient of $\PP(\FF)$ by this equivalence relation is the \textit{holonomy groupoid} $\HH\big(\pi_{B}^{k,\FF}\big)$. The~arguments of~\cite[Theorem~5.15]{mac3} then apply to show that if in particular $\pi_{B}$ admits invariant germs of sections about each point then we have a hierarchy
\begin{equation}\label{hierintro}
\begin{tikzcd}[row sep = large]
& & & \HH\big(\pi_{B}^{\g,\FF}\big) \ar[d,"\Pi_{B}^{\infty,\g}"] & & & \\ & & & \HH\big(\pi_{B}^{\infty,\FF}\big) \ar[dl,"\Pi_{B}^{k+1,\infty}"] \ar[dr,"\Pi_{B}^{k,\infty}"'] \ar[drrr,"\Pi_{B}^{0,\infty}"'] & & &\\ & \cdots\ar[r] & \HH\big(\pi_{B}^{k+1,\FF}\big) \ar[rr,"\Pi_{B}^{k,k+1}"'] & & \HH\big(\pi_{B}^{k,\FF}\big) \ar[r] & \cdots \ar[r] & \HH\big(\pi_{B}^{0,\FF}\big)
\end{tikzcd}
\end{equation}
of diffeological holonomy groupoids associated to the tower
\begin{gather*}
\Gamma_{\g}(\pi_{B})\rightarrow\Gamma_{\infty}(\pi_{B})\rightarrow\cdots\rightarrow \Gamma_{1}(\pi_{B})\rightarrow B
\end{gather*}
of germ/jet bundles. Smaller hierarchies also exist when $\pi_{B}$ only admits invariant $k$-jets of~sections about each point, with all singularly foliated bundles admitting invariant $0$-jets of~sections about each point.

Following this, we prove in Theorem~\ref{agree} that in the case of a singular foliation $(M,\FF)$, with $M$ of dimension $n$, the holonomy groupoid $\HH\big(\pi_{M\times\RB^{n}}^{\g,\FF}\big)$ associated to the trivial singularly foliated bundle $M\times\RB^{n}\rightarrow M$ is the quotient of the Garmendia--Villatoro holonomy groupoid by an equivalence relation which identifies groupoid elements if and only if they act identically on first integrals. Thus, by~\cite[Theorem~5.5]{VilGar1}, our construction is a quotient of the well-known Androulidakis--Skandalis holonomy groupoid~\cite{iakovos2} in such cases. We~also show that our construction agrees with that of Garmendia--Villatoro for regular foliations. Section~\ref{sec4} is concluded by~defining a class of morphisms of singularly foliated bundles, and we prove in Theorem~\ref{functoriality} that the hierarchy \eqref{hierintro} of holonomy groupoids is functorial under such morphisms. The~paper is concluded in Section~\ref{sec5} with a discussion of open questions.

\section{Background and prerequisites}\label{sec2}

\subsection{Notational conventions}

All manifolds and fibre bundles are assumed to be smooth, Hausdorff, without boundary and connected, and all maps thereof assumed to be smooth. Given any manifold $M$, we use $\XF_{M}$ to denote the sheaf of smooth vector fields on $M$, and $C^{\infty}_{M}$ the sheaf of smooth, real-valued functions on $M$. If $\pi_{B}\colon B\rightarrow M$ is a fibre bundle, then we denote by $\Gamma_{\pi_{B}}$ the sheaf of sections of~$\pi_{B}$. Given a smooth map $f\colon M\rightarrow N$ of manifolds, and sheaves $\SS_{M}$ on $M$ and $\SS_{N}$ on $N$, we~denote by $f_{!}\SS_{M}$ the pushforward of $\SS_{M}$ and by $f^{!}\SS_{N}$ the pullback of $\SS_{N}$~\cite[p.~65]{hartshorne}. Given a (possibly time-dependent) vector field $X = X(t,-)$ on an open set $\OO$ in a manifold $M$, we use $\Fl^{X}$ to denote its \textit{flow}. That is, for $x\in\OO$ and $t_{0}\in\RB$, $t\mapsto\Fl^{X}_{t,t_{0}}(x)$ is the unique solution to the initial value problem
\begin{gather*}
\frac{\rm d}{{\rm d}s}\bigg|_{t}f(s;x) = X(t,f(t;x)),\qquad f(t_{0};x) = x
\end{gather*}
defined by the vector field $X$~\cite[p.~236]{leesmth}. For time-independent $X$, $\Fl^{X}_{t,t_{0}}$ depends only on $t-t_{0}$ and will be denoted simply by $\Fl^{X}_{t-t_{0}}$. Vector fields are assumed time-independent unless otherwise stated.

\subsection{Singular foliations}

We begin by recalling the standard sheaf-theoretic definition of a singular foliation.

\begin{Definition}\label{singdef}
Let $M$ be a smooth manifold. A~\textit{singular foliation} on $M$ is a subsheaf $\FF\subset \XF_{M}$ of $C^{\infty}_{M}$-modules on $M$ for which the following hold.
\begin{enumerate}\itemsep=0pt
\item $\FF$ is \textit{closed under Lie brackets} in the sense that for every open set $\OO$ of $M$, $\FF(\OO)$ is closed under the Lie bracket of vector fields.
\item $\FF$ is \textit{locally finitely generated} in the sense that for each $x\in M$, there exists an open neighbourhood $\OO_{x}$ of $x$ and a finite family $X_{1},\dots,X_{k}$ of elements of $\FF(\OO_{x})$ such that $\FF(\OO_{x})$ is the $C^{\infty}_{M}(\OO_{x})$-span of the $X_{i}$.
\end{enumerate}
\end{Definition}

One can alternatively describe a singular foliation of a manifold $M$ as a locally finitely generated submodule of the compactly supported vector fields on $M$ which is closed under Lie brackets, as in~\cite{iakovos2}. By~\cite[Remark 1.8]{GarZam} these definitions are equivalent. One of the most important facts regarding singular foliations is the Stefan--Sussmann integration theorem~\cite{stefan, sussmann}.

\begin{Theorem}
Let $M$ be a manifold with a singular foliation $\FF$. Then $\FF$ integrates to give a~decom\-position of $M$ into smoothly immersed submanifolds called leaves.
\end{Theorem}

The following important theorem, due to Androulidakis and Skandalis~\cite[Theorem~0.1]{iakovos2}, says that the leaves of a singular foliation always arise as the orbits of a certain topological groupoid, called the \textit{holonomy groupoid} of the foliation.

\begin{Theorem}\label{integrate}
Let $M$ be a manifold with a singular foliation $\FF$. Then there exists a topological groupoid $\HH(\FF)$ for which the following hold.
\begin{enumerate}\itemsep=0pt
\item[$1.$] $\HH(\FF)$ \textit{integrates} the foliation $\FF$, in the sense that its orbits are the leaves of $\FF$.
\item[$2.$] $\HH(\FF)$ is \textit{minimal} in the sense that if $G$ is any Lie groupoid which integrates $\FF$, there is an~open subgroupoid $G_{0}$ of $G$ and a surjective morphism $G_{0}\rightarrow\HH(\FF)$ of topological groupoids.
\end{enumerate}
The groupoid $\HH(\FF)$ is called the \textit{holonomy groupoid} of $\FF$.
\end{Theorem}

Androulidakis and Skandalis built their groupoid using ``bisubmersions'' defined by iterated flows of the vector fields defining the foliation. Recent work by Garmendia and Villatoro~\cite{VilGar1} recovers the Androulidakis--Skandalis holonomy groupoid as a diffeological quotient of a certain path space. It is the Garmendia--Villatoro construction that most closely resembles the construction we give in this article.

A very large family of examples of singular foliations is furnished by \textit{Lie algebroids} as we now describe.

\begin{Example}
A \textit{Lie algebroid} consists of a (finite rank) vector bundle $\pi_{A}\colon A\rightarrow M$ over a manifold $M$ together with a Lie bracket on the space $\Gamma(M;A)$ of its smooth sections and a morphism $\rho\colon A\rightarrow TM$ of vector bundles known as the \textit{anchor map} which preserves Lie brackets on sections. Then the image of the compactly supported smooth sections of $A$ under~$\rho$ is a singular foliation~\cite[Example 1.3]{iakovos2}. Of particular importance are Lie algebroids that are \textit{integrable} in the sense that they integrate to Lie groupoids. The~characterisation of which Lie algebroids are integrable in this sense was solved by Crainic and Fernandes~\cite{cf2}. Lie algebroids arise in many geometric situations.
\begin{enumerate}\itemsep=0pt
\item Lie algebras, which are Lie algebroids over a point and are integrated by Lie groups by Lie's third theorem.
\item Regular foliations, defined by Lie algebroids with \textit{injective} anchor map. Such foliations are integrated by their Winkelnkemper--Phillips holonomy groupoids~\cite{holimper,wink}.
\item Debord foliations, associated to algebroids whose anchor maps are injective on a dense set. The~integration problem for these foliations was solved by Debord in~\cite{debord, debord2}.
\item Poisson manifolds $M$, whose Lie algebroids $T^{*}M\rightarrow M$ are equipped with a Lie bracket arising from the Poisson structure. Poisson manifolds admit (and are characterised by) singular foliations by symplectic leaves~\cite[p.~113]{cf2}.
\end{enumerate}
\end{Example}

We will not be making use of Lie algebroids in this article. Like the constructions of~And\-rou\-lidakis--Skandalis and Garmendia--Villatoro, our constructions will be founded on the more general Definition~\ref{singdef}.

\subsection{Jet bundles and prolongation}

We recall in this subsection some well-known theory of jet bundles, drawn primarily from~\cite{varbi, saunders}. Although the reader is likely familiar with this theory already, we include the following outline both to introduce our rather unconventional notation (which we choose for consistency with the bundles of germs to be introduced in Section~\ref{pseudobundlesofgerms}) and to point out the structures that will be~most relevant in our constructions.

\begin{Definition}
Let $\pi_{B}\colon B\rightarrow M$ be a fibre bundle, with fibre $F$, and let $k\geq0$. We~say that two local sections $\sigma_{1}$ and $\sigma_{2}$ of $\pi_{B}$ defined in a neighbourhood of $x\in M$ have the same~$k$-\textit{jet at}~$x$ if $\sigma_{1}(x) =b= \sigma_{2}(x)$, and for any local coordinate neighbourhood $\OO\cong\OO_{M}\times\OO_{F}$ of $b$, one~has
\begin{gather*}
\frac{\partial^{|I|}\sigma^{i}_{1}}{\partial x^{I}}(x) = \frac{\partial^{|I|}\sigma^{i}_{2}}{\partial x^{I}}(x)
\end{gather*}
for all $i = 1,\dots,\dim(F)$, and all multi-indices $I$ with $|I|\leq k$. Having the same $k$-jet at a point~$x$ is an equivalence relation on the set of local sections defined about $x$, and we denote the $k$-jet equivalence class of any such local section by $[\sigma]^{k}_{x}$.
\end{Definition}

The $k$-jets of local sections fit into a fibre bundle in a natural way.

\begin{Definition}
Let $\pi_{B}\colon B\rightarrow M$ be a fibre bundle, and let $k\geq 0$. For each $x\in M$, denote the set of all $k$-jets of local sections defined near $x$ by $\Gamma_{k}(\pi_{B})_{x}$, and define
\begin{gather*}
\Gamma_{k}(\pi_{B}):=\bigsqcup_{x\in M}\Gamma_{k}(\pi_{B})_{x},
\end{gather*}
with $\pi_{B}^{k}\colon \Gamma_{k}(\pi_{B})\rightarrow M$ denoting the canonical projection. Define $\pi_{B}^{0,k}\colon \Gamma_{k}(\pi_{B})\rightarrow B$ by
\begin{gather*}
\pi_{B}^{0,k}\big(x,[\sigma]^{k}_{x}\big):=\sigma(x),
\end{gather*}
and observe that any choice of local coordinate trivialisation $(x^{i},f^{\alpha})$ on $\OO$ about a point $b\in B$ determines coordinates
\begin{gather*}
\big(x^{i},f^{\alpha},f^{\alpha}_{i},\dots,f^{\alpha}_{I}\big):=\bigg(x^{i},f^{\alpha},\frac{\partial f^{\alpha}}{\partial x^{i}},\dots,\frac{\partial^{|I|}f^{\alpha}}{\partial x^{I}}\bigg)
\end{gather*}
on the set $\big(\pi^{0,k}_{B}\big)^{-1}(\OO)$. These coordinates give the projection $\pi_{k}\colon \Gamma_{k}(\pi_{B})\rightarrow M$ the structure of~a~fibre bundle, called the $k^{\rm th}$ \textit{order jet bundle} of $\pi_{B}$.
\end{Definition}

The jet bundles of a fibre bundle $\pi_{B}\colon B\rightarrow M$ admit projections $\pi_{B}^{k,l}\colon \Gamma_{l}(\pi_{B})\rightarrow \Gamma_{k}(\pi_{B})$ defined by
\begin{gather*}
\pi_{B}^{k,l}\big([\sigma]^{l}_{x}\big):=[\sigma]^{k}_{x},\qquad [\sigma]^{l}_{x}\in \Gamma_{l}(\pi_{B})
\end{gather*}
for any $l\geq k$, and these projections form a projective system. The~projective limit of this system, denoted $\Gamma_{\infty}(\pi_{B})$, inherits a natural smooth structure as a projective limit of manifolds \cite[Chapter One, Section A]{varbi} (equivalently via the projective limit diffeology~\cite[Section 1.39]{diffeology}). The~space $\Gamma_{\infty}(\pi_{B})$ is usually identified with the set of $\infty$-jets of local sections of $\pi_{B}$, and admits a canonical projection $\pi_{B}^{\infty}\colon \Gamma_{\infty}(\pi_{B})\rightarrow M$ given by
\begin{gather*}
\pi_{B}^{\infty}\big([\sigma]^{\infty}_{x}\big):=x,\qquad [\sigma]^{\infty}_{x}\in \Gamma_{\infty}(\pi_{B}).
\end{gather*}
One therefore obtains a hierarchy of jet bundles
\begin{gather*}
\Gamma_{\infty}(\pi_{B})\rightarrow\cdots\xrightarrow{\pi_{B}^{k,k+1}} \Gamma_{k}(\pi_{B})\xrightarrow{\pi_{B}^{k-1,k}}\cdots\xrightarrow{\pi_{B}^{1,2}} \Gamma_{1}(\pi_{B})\xrightarrow{\pi_{B}^{0,1}}B\xrightarrow{\pi_{B}}M.
\end{gather*}

Since we will be concerned primarily with singular foliations, which arise from families of~vec\-tor fields, we will need to know about vector fields on jet bundles. A~particularly important class of vector fields on fibre bundles, in which we will be primarily interested, is those that are \textit{projectable} in the following sense.

\begin{Definition}
Let $\pi_{B}\colon B\rightarrow M$ be a fibre bundle. A~vector field $X$ on $B$ is said to be \textit{pro\-jec\-table} if there is a vector field $(\pi_{B})_{*}(X)$ on $M$ for which
\begin{gather}\label{projectableeqn}
{\rm d}\pi_{B}\circ X = (\pi_{B})_{*}(X)\circ \pi_{B}
\end{gather}
on all of $B$. We~denote by $\XF_{\proj}(B)$ the set of projectable vector fields on $B$.
\end{Definition}

Projectable vector fields on a bundle prolong in a natural way to vector fields on the associated jet bundles.

\begin{Definition}
Let $\pi_{B}\colon B\rightarrow M$ be a fibre bundle, and let $X\in\XF_{\proj}(B)$ be a (possibly time-dependent) projectable vector field. For $1\leq k<\infty$, the \textit{$k$-jet prolongation} of $X$ is the vector field $\pf^{k}(X)\in\XF(\Gamma_{k}(\pi_{B}))$ defined by
\begin{gather*}
\pf^{k}(X)\big(x,[\sigma]^{k}_{x}\big):=\frac{\rm d}{{\rm d}t}\bigg|_{t=0}\Big(\Fl^{(\pi_{B})_{*}X}_{t,0}(x),\big[\Fl^{X}_{t,0}\circ \sigma\circ\Fl^{(\pi_{B})_{*}X}_{0,t}\big]^{k}_{\Fl^{(\pi_{B})_{*}X}_{t,0}(x)}\Big)
\end{gather*}
for all $[\sigma]^{k}_{x}\in \Gamma_{k}(\pi_{B})$. We~denote by $\XF_{\proj}(\Gamma_{k}(\pi_{B}))$ the image of $\pf^{k}$.
\end{Definition}

It will be useful later to note here that if $X$ is a time-independent projectable vector field on~a~bundle $\pi_{B}$, it is given in coordinates by $X = a^{i}\frac{\partial}{\partial x^{i}} + b^{\alpha}\frac{\partial}{\partial f^{\alpha}}$, where $a^{i}$ are smooth functions depending only on the $x^{i}$ while the $b^{\alpha}$ depend on both the $x^{i}$ and the $f^{\alpha}$. Then for any $1\leq k<\infty$, the $k$-jet prolongation $\pf^{k}(X)$ of $X$ is given in coordinates over a point $(x^{i},f^{\alpha},f^{\alpha}_{i},\dots)\in \Gamma_{k}(\pi_{B})$ by the formula (cf.~\cite[Theorem~2.36]{olver})
\begin{gather}\label{prolongation}
\pf^{k}(X) = a^{i}D_{i}^{(k)}+\sum_{|I|=k}\!\bigg(D^{(k)}_{I}b^{\alpha}-\sum_{J\subset I}\big(D^{(k)}_{I\setminus J}a^{i}\big)f^{\alpha}_{Ji}\bigg)\frac{\partial}{\partial f^{\alpha}_{I}}+\!\sum_{|I|=0}^{k-1}D^{(k)}_{I}(b^{\alpha}-a^{i}f^{\alpha}_{i})\frac{\partial}{\partial f^{\alpha}_{I}},
\end{gather}
where the sum over $J\subset I$ is a sum over all \textit{strict} subsets of the multi-index $J$, and where we have absorbed the constants arising from the symmetry of mixed partial derivatives into our notation as in~\cite[equation~(1.15)]{varbi}. Here $D^{(k)}_{I} = D^{(k)}_{i_{1}}\cdots D^{(k)}_{i_{k}}$, where $D^{(k)}_{i}$ is the \textit{total derivative operator} defined by $D^{(k)}_{i} = \frac{\partial}{\partial x^{i}}+\sum_{|I|=0}^{k-1}f^{\alpha}_{Ii}\frac{\partial}{\partial f^{\alpha}_{I}}$.

The argument of Proposition~\ref{germprol} can be used to deduce the following fact, which justifies our choice of notation in denoting the image of $\pf^{k}$ in $\XF(\Gamma_{k}(\pi_{B}))$ by $\XF_{\proj}(\Gamma_{k}(\pi_{B}))$.

\begin{Proposition}
Let $\pi_{B}\colon B\rightarrow M$ be a fibre bundle, and let $X$ be a projectable vector field on~$B$. Then the $\pf^{k}(X)$ are a projectable family, in the sense that
\begin{gather*}
{\rm d}\pi_{B}^{l,k}\circ \pf^{k}(X) = \pf^{l}(X)\circ\pi_{B}^{l,k}
\end{gather*}
on $\Gamma_{k}(\pi_{B})$ for all $l\leq k$.
\end{Proposition}

Projectability of the $\pf^{k}(X)$ guarantees (cf.~\cite[equation~(1.11)]{varbi}) that pointwise they admit a~pro\-jective limit, defining a section $\pf^{\infty}(X)$ of the projective limit tangent bundle of $\Gamma_{\infty}(\pi_{B})$ (cf.~\cite[Chapter One, Section B]{varbi}) for which the following holds.

\begin{Proposition}\label{towerfield1}
Let $\pi_{B}\colon B\rightarrow M$ be a fibre bundle. Then for each $l\leq k\leq\infty$ there is a~bracket-preserving homomorphism $\big(\pi_{B}^{l,k}\big)_{*}\colon \XF_{\proj}(\Gamma_{k}(\pi_{B}))\rightarrow\XF_{\proj}(\Gamma_{l}(\pi_{B}))$ and the diagram
\begin{center}
\begin{tikzcd}[column sep = huge]
& \XF_{\proj}(\Gamma_{\infty}(\pi_{B})) \ar[d]
\\
& \vdots \ar[d,"\left(\pi_{B}^{1,2}\right)_{*}"]
\\
& \XF_{\proj}(\Gamma_{1}(\pi_{B})) \ar[d,"\left(\pi_{B}^{0,1}\right)_{*}"]
\\
\XF_{\proj}(B) \ar[uuur,"\pf^{\infty}"] \ar[ur,"\pf^{1}"] \ar[r,"\id"] & \XF_{\proj}(B)
\end{tikzcd}
\end{center}
commutes.
\end{Proposition}

\subsection{Diffeology}

We recall in this subsection some basic objects of study in diffeology that will be relevant for our constructions. The~most comprehensive reference on diffeology is the wonderful book~\cite{diffeology} by P. Iglesias-Zemmour.

\begin{Definition}\label{diffeology}
A function $\varphi\colon U\rightarrow \mathcal{X}$ from an open subset $U$ of some finite-dimensional Euclidean space to a set $\XX$ is called a \textit{parametrisation}. A~\textit{diffeology} on a set $\XX$ is a family D of~para\-metrisations satisfying the following axioms.
\begin{enumerate}\itemsep=0pt
\item The family D contains all constant parametrisations.
\item If $\varphi\colon U\rightarrow \XX$ is a parametrisation such that every point $u\in U$ has an open neighbourhood $V\subset U$ for which $\varphi|_{V}$ is an element of D, then $\varphi$ itself is an element of D.
\item For every element $\varphi\colon U\rightarrow \XX$ of D, every open set $V$ of any finite-dimensional Euclidean space, and for every smooth function $f\colon V\rightarrow U$, the composite $\varphi\circ f\colon V\rightarrow \XX$ is contained in D.
\end{enumerate}
A set with a diffeology is called a \textit{diffeological space}, and the elements of the diffeology are called its \textit{plots}. If $\XX$ and $\YY$ are two diffeological spaces, then a function $f\colon \XX\rightarrow \YY$ is said to be \textit{smooth} if for every plot $\varphi\colon U\rightarrow \XX$ of $\XX$, the composite $f\circ\varphi\colon U\rightarrow \YY$ is a plot of $\YY$. A~smooth bijection of diffeological spaces is said to be a \textit{diffeomorphism} if it is smooth with smooth inverse.
\end{Definition}

Every manifold is a diffeological space, with diffeology constituted by the set of all parametrisations that are smooth in the usual sense. Moreover a map between manifolds is smooth in the manifold sense if and only if it is smooth in the diffeological sense. Thus the category of~mani\-folds and smooth maps is a full and faithful subcategory of the category of diffeological spaces and smooth maps.

Diffeologies can be pushed forward and pulled back by functions of sets. This fact will be invoked frequently for our constructions.

\begin{Definition}
Let $\XX$ and $\YY$ be sets, and let $f\colon \XX\rightarrow \YY$ be a function.
\begin{enumerate}\itemsep=0pt
\item If $\XX$ has a diffeology, then the \textit{pushforward diffeology induced by $f$} is defined by declaring a parametrisation $\varphi\colon U\rightarrow \YY$ to be a plot if and only if every $u\in U$ has an open neighbourhood $V\subset U$ such that either $\varphi|_{V}$ is constant, or equal to the composite $f\circ\psi$ for some plot $\psi\colon V\rightarrow \XX$ of $\XX$. The~map $f$ is said to be a \textit{subduction} if it is surjective and if $\YY$ is equipped with the pushforward diffeology induced by $f$.
\item If $\YY$ has a diffeology, then the \textit{pullback diffeology induced by $f$} is defined by declaring a~parametrisation $\psi\colon U\rightarrow \XX$ to be plot if and only if the composite $f\circ\psi\colon U\rightarrow \YY$ is a~plot of~$\YY$. The~map $f$ is said to be an \textit{induction} if it is injective and if $\XX$ is equipped with the pullback diffeology from $\YY$.
\end{enumerate}
The following special cases are of particular importance. Let $\XX$ be a diffeological space.
\begin{enumerate}\itemsep=0pt
\item If $\sim$ is any equivalence relation on $\XX$, then the \textit{quotient diffeology} $\XX/{\sim}$ is the pushfoward diffeology arising from the quotient $\XX\rightarrow \XX/{\sim}$.
\item If $S$ is any subset of $\XX$, then the \textit{subspace diffeology} on $S$ is the pullback diffeology arising from the inclusion $S\hookrightarrow \XX$.
\item If $\YY$ is any other diffeological space, then the \textit{product diffeology} on $\XX\times \YY$ is the smallest diffeology for which the projections onto the factors are subductions.
\end{enumerate}
Quotients, subspaces and products will always be assumed to be equipped with the respective diffeologies defined above unless otherwise stated.
\end{Definition}

One of the features of the category of diffeological spaces is that the set of all morphisms between any two objects in the category is itself an object.

\begin{Definition}
Let $\XX$ and $\YY$ be diffeological spaces, and denote by $C^{\infty}(\XX,\YY)$ the set of~all smooth maps $f\colon \XX\rightarrow \YY$. The~\textit{functional diffeology} on $C^{\infty}(\XX,\YY)$ is defined by declaring a~para\-metrisation $\tilde{f}\colon U\rightarrow C^{\infty}(\XX,\YY)$ to be a plot if and only if the associated map $U\times\XX\ni(u,x)\mapsto\tilde{f}(u)(x)\rightarrow\YY$ is smooth.
\end{Definition}

The familiar notion of a fibre bundle over a manifold has a far reaching generalisation to diffeological spaces. It is the flexibility afforded by this generalisation that permits most of the constructions in this paper.

\begin{Definition}
A \textit{diffeological pseudo-bundle} is a subduction $\pi_{\BB}\colon \BB\rightarrow \XX$ of diffeological spaces. Such a pseudo-bundle is in particular called a \textit{diffeological vector pseudo-bundle} if each fibre of $\pi_{\BB}$ is a vector space for which the vector space operations are smooth with respect to the subspace diffeology, and such that the zero section is smooth.
\end{Definition}

A diffeological pseudo-bundle need not have fibres that are all diffeomorphic, and of course need not be locally trivial in any sense (see for instance Example~\ref{germfol}). An important subclass of~diffeological pseudo-bundles are \textit{diffeological bundles}, which have mutually diffeomorphic fibres and which are locally trivial under pullbacks by plots. Diffeological bundles are distinguished by the behaviour of their \textit{structure groupoids}. Before we give the definition of this object, let us record what we mean by diffeological categories and groupoids.

\begin{Definition}
Let $\CC$ be a small category, with object set identified as a subset of the morphism set via the map which sends each object to its associated identity morphism. We~say that~$\CC$ is a \textit{diffeological category} if its set of morphisms is equipped with a diffeology for which the range, source, and composition are all smooth. If $\CC$ is in addition a groupoid, whose inversion map is smooth, we call $\CC$ a \textit{diffeological groupoid}.
\end{Definition}

Let now $\pi_{\BB}\colon \BB\rightarrow \XX$ be a smooth surjection of diffeological spaces. Denote by $\Aut(\pi_{\BB})$ the groupoid with object set $\XX$, and with morphisms from $x$ to $y$ constituted by the set $\Diff(\BB_{x},\BB_{y})$ of all diffeomorphisms from the fibre $\BB_{x}$ over $x$ to the fibre $\BB_{y}$ over $y$, with the obvious range, source, inversion and composition. The~groupoid $\Aut(\pi_{\BB})$ admits a smallest diffeology, called the \textit{functional diffeology}, under which the evaluation map $\ev\colon \Aut(\pi_{\BB})\times_{s,\pi_{\BB}}\BB\ni(f,b)\mapsto f(b)\in \BB$ is smooth, and under which $\Aut(\pi_{\BB})$ is a diffeological groupoid (see~\cite[Section~8.7]{diffeology} for details).

\begin{Definition}\label{structuregpd}
 Let $\pi_{\BB}\colon \BB\rightarrow \XX$ be a surjection of diffeological spaces. Equipped with the functional diffeology, we refer to $\Aut(\pi_{\BB})$ as the \textit{structure groupoid} of the surjection $\pi_{\BB}$. We~say that $\pi_{\BB}$ is a \textit{diffeological bundle} if the characteristic map $(r,s)\colon \Aut(\pi_{\BB})\rightarrow \XX\times \XX$ is a sub\-duction.
\end{Definition}

Diffeological bundles, unlike general diffeological pseudo-bundles, have a typical fibre to which all other fibres are diffeomorphic, and the pullback of a diffeological bundle along any plot is locally trivial~\cite[p.~240]{diffeology}.

By definition, singular foliations arise from certain families of sections of tangent bundles. To~use diffeology to study singular foliations, therefore, we need a notion of tangent bundle for a diffeological space. A~number of definitions have been proposed for this purpose, which, while coincident for manifolds, do not coincide for general diffeological spaces (see~\cite{cw} for a detailed discussion). The~point of view that we find useful here, as in~\cite{mac3}, is that of \textit{internal} tangent spaces and bundles. The~paper~\cite{cw} provides a categorical definition of internal tangent spaces based on work of Hector~\cite{hector}, which may be summarised as follows.

\begin{Definition}\label{tangentbundle}
Let $\XX$ be a diffeological space and let $x\in \XX$. Denote by $p_{x}$ the set of all \textit{plots centered at $x$}, that is, plots $\varphi\colon U\rightarrow \XX$ such that $0\in U$ and $\varphi(0) = x$. Denote by $T_{0}\dom(\varphi)$ the tangent space at zero of the domain of any such plot, and by $\vec{v}_{\varphi}$ the image of $\vec{v}\in T_{0}\dom(\varphi)$ in the direct sum $\bigoplus_{\varphi\in p_{x}}T_{0}\dom(\varphi)$. The~\textit{internal tangent space of $\XX$ at $x$} is the quotient space~$T_{x}\XX$ of the direct sum
\begin{gather*}
\bigoplus_{\varphi\in p_{x}} T_{0}\dom(\varphi)
\end{gather*}
by the subspace generated by all vectors of the form $f_{*}(\vec{v})_{\varphi} - \vec{v}_{\varphi\circ f}$, where $\varphi\in p_{x}$, $f\colon \dom(f)\rightarrow\dom(\varphi)$ is any smooth map of open Euclidean domains with $0\in\dom(f)$ and $f(0) = 0$, and with $\vec{v}\in T_{0}\dom(f)$. The~class of an element $\vec{v}_{\varphi}$, $\vec{v}\in T_{0}\dom(\varphi)$, will be denoted $\varphi_{*}(\vec{v})$. By~\cite[Proposition~3.3]{cw}, elements of $T_{x}\XX$ can always be written as linear combinations of $\varphi_{*}({\rm d}/{\rm d}t)$ for 1-plots $\varphi\colon (-\epsilon,\epsilon)\rightarrow\XX$, which we will usually write as ${\rm d}/{\rm d}t|_{0}\varphi_{t}$.
\end{Definition}

Let $\XX$ be a diffeological space, and consider the set $T\XX:=\bigsqcup_{x\in \XX}T_{x}\XX$. For any plot \mbox{$\varphi\colon U\rightarrow \XX$} and for any $u\in U$, denote by $\tau_{u}\colon v\mapsto v+u$ the $u$-translation map, so that $\tau_{u}^{-1}(U)$ contains $0$ and $\varphi\circ\tau_{u}\colon \tau_{u}^{-1}(U)\rightarrow \XX$ is a plot centered at $\varphi(u)$. Define ${\rm d}\varphi\colon TU\rightarrow T\XX$ by the formula
\begin{gather*}
{\rm d}\varphi(u,\vec{v}):=(\varphi(u),(\varphi\circ\tau_{u})_{*}(\vec{v})),\qquad (u,\vec{v})\in TU.
\end{gather*}
These maps were first considered by Hector~\cite{hector}. Then there exists a smallest diffeology on $T\XX$, called the \textit{dvs diffeology}~\cite{cw}, for which the natural projection $\pi_{T\XX}\colon T\XX\rightarrow \XX$ is a diffeological vector pseudo-bundle, and which contains the parametrisations ${\rm d}\varphi\colon TU\rightarrow T\XX$ as plots.

\begin{Definition}
Let $\XX$ be a diffeological space. The~diffeological vector pseudo-bundle $\pi_{T\XX}$: $T\XX\rightarrow \XX$ is called the \textit{internal tangent bundle of $\XX$}.
\end{Definition}

The internal tangent bundle is functorial under smooth maps of diffeological spaces.

\begin{Definition}\label{differential}
Let $f\colon \XX\rightarrow \YY$ be a smooth map of diffeological spaces. For any plot $\varphi$ of $\XX$ centered at $x$, define
\begin{gather*}
{\rm d}f(x,\varphi_{*}(\vec{v})):=(f(x),(f\circ\varphi)_{*}(\vec{v})),\qquad x\in \XX
\end{gather*}
and extend by linearity to a map ${\rm d}f\colon T\XX\rightarrow T\YY$. Then ${\rm d}f$ is a smooth map~\cite[Proposition~4.8]{cw} called the \textit{pushforward} or \textit{differential} of $f$.
\end{Definition}

Finally, we recall that every diffeological space admits a natural topology with respect to which all plots are continuous.

\begin{Definition}\label{Dtop}
Let $\XX$ be a diffeological space. The~\textit{D-topology} on $\XX$ is the topology whose open sets are precisely those sets $A\subset \XX$ for which $\varphi^{-1}(A)$ is open for all plots $\varphi$ of $\XX$.
\end{Definition}

For manifolds, the D-topology coincides with the usual topology. Although the D-topology will not play a central role in any of our constructions, we will see in Section~\ref{relsheaf} that it gives the \'{e}tale space of any sheaf of sections of a fibre bundle a natural topology which is \textit{distinct} from the usual \'{e}tale topology.

\section{Diffeological constructions}\label{sec3}

\subsection{Pseudo-bundles of germs}\label{pseudobundlesofgerms}

In~\cite{mac3}, we introduced ``bundles of germs'' of sections of certain fibre bundles. This construction can be generalised easily as follows. Let $\pi_{B}\colon B\rightarrow M$ be a smooth fibre bundle over a smooth manifold $M$, and let $\SS$ be a sheaf of smooth sections of $\pi_{B}$. Assume that $\SS$ is \textit{locally nonempty}, in the sense that for each $x\in M$, we can find an open neighbourhood $\mathcal{O}$ of $x$ such that $\SS(\mathcal{O})$ is nonempty. Define the \textit{total space} $\SS_{\loc}$ of the sheaf $\SS$ as the union
\begin{gather*}
\SS_{\loc}:=\bigcup_{\mathcal{O}}\SS(\mathcal{O})
\end{gather*}
over all open sets $\mathcal{O}$ in $M$ of elements of $\SS(\mathcal{O})$. Thus $\SS_{\loc}$ is the set of all locally defined sections of $\pi_{B}$ that belong to some $\SS(\mathcal{O})$. The~arguments of~\cite[Section~1.63]{diffeology} can be used to show that a diffeology may be defined on $\SS_{\loc}$ as follows.

\begin{Proposition}\label{functprop}
Let $M$ be a manifold, and let $\SS$ be a sheaf of smooth sections of some fibre bundle $B$ over $M$. Declare a parametrisation $\tilde{\sigma}\colon U\rightarrow \SS_{\loc}$ to have the property \textit{funct} if for all $u_{0}\in U$ and $x_{0}\in\dom(\tilde{\sigma}(u_{0}))$, there exists an open neighbourhood $V\subset U$ of $u_{0}$ and an open neighbourhood $\OO\subset \dom(\tilde{\sigma}(u_{0}))$ of $x_{0}$ for which $\OO\subset\dom(\tilde{\sigma}(u))$ for all $u\in V$ and for which the map $V\times\OO\ni(u,x)\mapsto\tilde{\sigma}(u)(x)\in B$ is smooth. Then the collection of parametrisations with the property \textit{funct} defines a diffeology on $\SS_{\loc}$.
\end{Proposition}

\begin{proof}
Observe first that if $\tilde{\sigma}\colon U\ni u\mapsto \sigma\in \SS(\mathcal{O}')$ is a constant plot, then we can simply take $\OO:=\OO'$ and $V = U$ to see that $\tilde{\sigma}$ has property \textit{funct}. Therefore axiom 1 of Definition~\ref{diffeology} is satisfied. To~see that axiom 2 is satisfied, suppose that $\tilde{\sigma}\colon U\rightarrow\SS_{\loc}$ is a parametrisation such that for each $u_{0}\in U$, there exists an open neighbourhood $V$ of $u_{0}$ such that $\tilde{\sigma}|_{V}$ has property \textit{funct}. Then by definition $\tilde{\sigma}$ must itself have property \textit{funct}, so that axiom 2 is satisfied. Finally, to prove that axiom 3 is satisfied, suppose that $\tilde{\sigma}\colon U\rightarrow\SS_{\loc}$ has property \textit{funct} and that $U'$ is some open Euclidean domain with $\varphi\colon U'\rightarrow U$ a smooth function. Let $u_{0}'\in U'$ and denote $u_{0}:=\varphi(u_{0}')\in U$. Let $V$ be an open neighbourhood of $u_{0}$ in $U$, and let $\OO$ be an open subset of $\dom(\tilde{\sigma}(u_{0}))$ for which $\OO\subset\dom(\tilde{\sigma}(u))$ for all $u\in V$ and for which $V\times\OO\ni(u,x)\mapsto\tilde{\sigma}(u)(x)\in B$ is smooth. Then $V':=\varphi^{-1}(V)$ is an open neighbourhood of $u_{0}'$ in $U'$, the open set $\OO$ satisfies $\OO\subset\dom(\tilde{\sigma}\circ\varphi(u))$ for all $u\in U$ and $V'\times\OO\ni(u',x)\mapsto\tilde{\sigma}(\varphi(u'))(x)\in B$ is smooth. Therefore axiom 3 is satisfied also.
\end{proof}

\begin{Definition}\label{funl}
Let $M$ be a manifold, and $\SS$ a sheaf of smooth sections of some fibre bundle over $M$. Then the diffeology on $\SS_{\loc}$ given in Proposition~\ref{functprop} is called the \textit{functional diffeology} on $\SS_{\loc}$.
\end{Definition}

Let us now consider the diffeological subspace
\begin{gather*}
\Gamma_{\loc}(\SS):=\{(x,\sigma)\colon x\in\dom(\sigma)\}
\end{gather*}
of the diffeological product $M\times \SS_{\loc}$. There is then clearly a surjective, smooth map $\pi_{\Gamma_{\loc}(\SS)}$: $\Gamma_{\loc}(\SS)\rightarrow M$ defined by
\begin{gather*}
\pi_{\Gamma_{\loc}(\SS)}(x,\sigma):=x,\qquad (x,\sigma)\in\Gamma_{\loc}(\SS),
\end{gather*}
which is moreover a subduction. Indeed, for any plot $\tilde{x}\colon U\rightarrow M$ and for any $u\in U$ we can always find an open neighbourhood $V$ of $u$ in $U$ such that $\tilde{x}(V)$ is contained in some open neighbour\-hood~$\mathcal{O}$ of $\tilde{x}(u)$ for which $\SS(\mathcal{O})$ contains some element $\sigma$. Now defining $\rho\colon V\rightarrow \Gamma_{\loc}(\SS)$ by~simply
\begin{gather*}
\rho(v):=(\tilde{x}(v),\sigma),\qquad v\in V
\end{gather*}
we have that $\pi_{\Gamma_{\loc}(\SS)}\circ\rho = \tilde{x}$, making $\pi_{\Gamma_{\loc}(\SS)}$ a subduction as claimed. The~fibre $\Gamma_{\loc}(\SS)_{x}$ over any $x\in M$ is the nonempty space consisting of sections $\sigma$ of $\SS$ defined on some open neighbourhood of $x$, equipped with the functional diffeology of Definition~\ref{funl}. The~subduc\-tion~$\pi_{\Gamma_{\loc}(\SS)}$ is the first step on the way to defining a genuinely useful object. Our next example shows why $\pi_{\Gamma_{\loc}(\SS)}$ is too large to be of much use in its own right.

\begin{Example}\label{ex1}
Let $\FF$ be a singular foliation of $M$, and denote the corresponding sheaf by the same symbol. Each fibre $\Gamma_{\loc}(\FF)_{x}$ is then \textit{almost} a Lie algebra. Indeed, if~$X\in\FF(\OO_{1})$ and \mbox{$Y\in\FF(\OO_{2})$} contain~$x$ in their domains of definition, then on the open neighbourhood $\OO:=\OO_{1}\cap \OO_{2}$ of~$x$, the Lie bracket $[X,Y]\in\FF(\OO)$ is defined. There is, however, nothing special about the choice $\OO:=\OO_{1}\cap\OO_{2}$, and indeed $[X,Y]$ also makes sense on any open neighbourhood~$\OO'$ of~$x$ within~$\OO$ and technically defines a distinct element $[X,Y]|_{\OO'}\in\FF(\OO')$. One encounters essentially the same problem when trying to define vector space operations in $\Gamma_{\loc}(\FF)_{x}$. To~rectify this sort of problem we work instead with a quotient of $\Gamma_{\loc}(\FF)$.
\end{Example}

Let us again return to a sheaf $\SS$ of sections of a bundle $\pi_{B}\colon B\rightarrow M$. Let us denote the germ at $x$ of any local section $\sigma$ of $\pi_{B}$ defined in an open neighbourhood of $x$ by $[\sigma]_{x}^{\g}$. We~define an equivalence relation~$\sim_{\g}$ on $\Gamma_{\loc}(\SS)$ by declaring $(x,\sigma)\sim_{\g}(y,\eta)$ if and only if $x = y$ and $[\sigma]^{\g}_{x} = [\eta]^{\g}_{x}$. We~denote by $\Gamma_{\g}(\SS)$ the diffeological quotient of $\Gamma(\SS)$ by the equivalence relation~$\sim_{\g}$, and denote by $\pi_{\SS}^{\g}\colon \Gamma_{\g}(\SS)\rightarrow M$ the obvious surjection
\begin{gather*}
\pi_{\SS}^{\g}\big(x,[\sigma]^{\g}_{x}\big):=x,\qquad \big(x,[\sigma]^{\g}_{x}\big)\in\Gamma_{\g}(\SS).
\end{gather*}
Since both the quotient map $q\colon \Gamma_{\loc}(\SS)\rightarrow\Gamma_{\g}(\SS)$ and the projection $\Gamma_{\loc}(\SS)\rightarrow M$ are subductions, so too is the projection $\pi_{\SS}^{\g}\colon \Gamma_{\g}(\SS)\rightarrow M$. Thus $\pi_{\SS}^{\g}\colon \Gamma_{\g}(\SS)\rightarrow M$ is a diffeological pseudo-bundle.

\begin{Definition}
Let $\SS$ be a sheaf of sections of a fibre bundle $\pi_{B}\colon B\rightarrow M$. Then the subduction $\pi_{\SS}^{\g}\colon \Gamma_{\g}(\SS)\rightarrow M$ is called the \textit{pseudo-bundle of germs of $\SS$}. If in particular $\SS$ is the sheaf of all sections of $\pi_{B}$, then we denote the pseudo-bundle of germs of $\SS$ by simply $\pi_{B}^{\g}\colon \Gamma_{\g}(\pi_{B})\rightarrow M$.
\end{Definition}

In fact the pseudo-bundle of germs of the full sheaf of sections of a fibre bundle $\pi_{B}\colon B\rightarrow M$ is a diffeological bundle, in the sense that all its fibres are isomorphic to the single diffeological space $C_{\g,0}^{\infty}\big(\RB^{\dim(M)};F\big)$ of germs at zero of smooth functions from $\RB^{\dim(M)}$ into the typical fibre~$F$ of~$B$. This can be seen, for instance, using an associated bundle construction in a similar fashion to~\cite[Remark 5.7]{mac3}~-- one need only replace the ``distinguished functions'' considered therein with coordinate maps on $M$. Thus it is entirely reasonable to refer to $\Gamma_{\g}(\pi_{B})$ as the \textit{bundle} of germs of sections of $\pi_{B}$. Let us now study its relationship with the jet bundles of $\pi_{B}$.

For each $k\leq\infty$ we have a canonical projection $\pi_{B}^{k,\g}\colon \Gamma_{\g}(\pi_{B})\rightarrow \Gamma_{k}(\pi_{B})$ onto the $k^{\rm th}$ order jet bundle of $\pi_{B}$ defined by
\begin{gather*}
\pi_{B}^{k,\g}\big(x,[\sigma]^{\g}_{x}\big):=\big(x,[\sigma]^{k}_{x}\big),\qquad \big(x,[\sigma]^{\g}_{x}\big)\in\Gamma_{\g}(\pi_{B}).
\end{gather*}
The arguments of~\cite[Proposition~5.14]{mac3} show that these projections are smooth, and are compatible with the jet projections $\pi_{B}^{l,k}\colon \Gamma_{k}(\pi_{B})\rightarrow \Gamma_{l}(\pi_{B})$ in the sense that $\pi_{B}^{l,\g} = \pi_{B}^{l,k}\circ\pi_{B}^{k,\g}$ for all $l\leq k$. We~therefore have a tower
\begin{center}
\begin{tikzcd}[row sep = large]
& & & \Gamma_{\g}(\pi_{B}) \ar[d,"\pi_{B}^{\infty,\g}"] & & &
\\
& & & \Gamma_{\infty}(\pi_{B}) \ar[dl,"\pi_{B}^{k+1,\infty}"] \ar[dr,"\pi_{B}^{k,\infty}"'] \ar[drrr,"\pi_{B}^{0,\infty}"'] & & &
\\
& \cdots\ar[r] & \Gamma_{k+1}(\pi_{B}) \ar[rr,"\pi_{B}^{k,k+1}"'] & & \Gamma_{k}(\pi_{B}) \ar[r] & \cdots \ar[r] & B
\end{tikzcd}
\end{center}
of diffeological bundles which extends the usual well-known tower of jet bundles of $\pi_{B}$.

Now there is not an easily identifiable Lie bracket on the space of vector fields (that is, sections of the internal tangent bundle) on $\Gamma_{\g}(\pi_{B})$, however there \textit{does} exist a certain diffeological subspace of vector fields which carries a natural Lie bracket. These vector fields are those that are contained in the image of a germinal prolongation operator from projectable vector fields on~$B$ to vector fields on $\Gamma_{\g}(\pi_{B})$.

\begin{Definition}
Let $\pi_{B}\colon B\rightarrow M$ be a fibre bundle. For (possibly time-dependent) $X\in\XF_{\proj}(B)$, the vector field $\pf^{\g}(X)$ on $\Gamma_{\g}(\pi_{B})$ defined by the formula
\begin{gather*}
\pf^{\g}(X)\big(x,[\sigma]^{\g}_{x}\big):=\frac{\rm d}{{\rm d}t}\bigg|_{t=0} \Big(\Fl^{(\pi_{B})_{*}X}_{t,0}(x),\big[\Fl^{X}_{t,0}\circ\sigma \circ\Fl^{(\pi_{B})_{*}(X)}_{0,t}\big]^{\g}_{\Fl^{(\pi_{B})_{*}(X)}_{t,0}(x)}\Big)
\end{gather*}
is called the \textit{germinal prolongation of $X$}. The~associated linear map $\pf^{\g}\colon \XF_{\proj}(B)\rightarrow \XF(\Gamma_{\g}(\pi_{B}))$ is called the \textit{germinal prolongation operator}, and its image is denoted $\XF_{\proj}(\Gamma_{\g}(\pi_{B}))$.
\end{Definition}

Our next result relates the germinal prolongation operator to the jet prolongations of projectable vector fields, and can be seen as a justification of the nomenclature ``germinal prolongation".

\begin{Proposition}\label{germprol}
Let $\pi_{B}\colon B\rightarrow M$ be a fibre bundle. Then for each $k\leq\infty$, we have
\begin{gather*}
{\rm d}\pi_{B}^{k,\g}\circ \pf^{\g}(X) = \pf^{k}(X)\circ\pi_{B}^{k,\g}
\end{gather*}
for all $($possibly time-dependent$)$ $X\in\XF_{\proj}(B)$. Consequently the tower of prolongations of~Pro\-po\-si\-tion~$\ref{towerfield1}$ completes to a tower
\begin{center}
\begin{tikzcd}[column sep = large, row sep = large]
& &\XF_{\proj}(\Gamma_{\g}(\pi_{B})) \ar[d, "(\pi_{B}^{\infty,\g})_{*}"]
\\
& & \XF_{\proj}(\Gamma_{\infty}(\pi_{B})) \ar[d]
\\
& & \vdots \ar[d,"\big(\pi_{B}^{0,1}\big)_{*}"]
\\
\XF_{\proj}(B) \ar[uuurr,"\pf^{\g}"] \ar[uurr,"\pf^{\infty}"] \ar[rr,"\id"] & & \XF_{\proj}(B)
\end{tikzcd}
\end{center}
\end{Proposition}

\begin{proof}
For any $k<\infty$ and $X\in\XF_{\proj}(B)$, we have
\begin{align*}
{\rm d}\pi_{B}^{k,\g}\circ\pf^{\g}(X)\big(x,[\sigma]^{\g}_{x}\big)
&= \frac{\rm d}{{\rm d}t}\bigg|_{t=0}\pi_{B}^{k,\g}\Big(\Fl^{(\pi_{B})_{*}X}_{t,0}(x),
\big[\Fl^{X}_{t,0}\circ\sigma\circ\Fl^{(\pi_{B})_{*}(X)}_{0,t}\big]^{\g}_{\Fl_{t,0}^{(\pi_{B})_{*}(X)}(x)}\Big)
\\
&= \frac{\rm d}{{\rm d}t}\bigg|_{t=0}\Big(\Fl^{(\pi_{B})_{*}X}_{t,0}(x), \big[\Fl^{X}_{t,0}\circ\sigma\circ\Fl^{(\pi_{B})_{*}(X)}_{0,t}\big]^{k}_{\Fl^{(\pi_{B})_{*}(X)}_{t,0}(x)}\Big)
\\
&= \pf^{k}(X)\big(\pi_{B}^{k,\g}\big(x,[\sigma]^{\g}_{x}\big)\big)
\end{align*}
for all $\big(x,[\sigma]^{\g}_{x}\big)\in\Gamma_{\g}(\pi_{B})$. For $k=\infty$ the result follows from the universal property of the projective limit of the $\pi_{B}^{k}$.
\end{proof}

The injectivity of the jet prolongation operators, which is apparent from equation~\eqref{prolongation} toge\-ther with Proposition~\ref{germprol}, implies that the germinal prolongation operator $\pf^{\g}\colon \XF_{\proj}(B)\rightarrow\XF(\Gamma_{\g}(\pi_{B}))$ of a bundle $\pi_{B}\colon B\rightarrow M$ is injective. Consequently, on $\XF_{\proj}(\Gamma_{\g}(\pi_{B}))$ we have a Lie bracket that is well-defined by the formula
\begin{gather*}
[\pf^{\g}(X),\pf^{\g}(Y)]:=\pf^{\g}([X,Y]),\qquad X,Y\in\XF_{\proj}(B),
\end{gather*}
with respect to which each $\big(\pi_{B}^{k,\g}\big)_{*}$ is a homomorphism of Lie algebras. While we will not be making use of this feature in this article, we remark that it distinguishes $\XF_{\proj}(\Gamma_{\g}(\pi_{B}))$ as a~rather special subspace of $\XF(\Gamma_{\g}(\pi_{B}))$, which, like the vector fields of many other diffeological spaces~\cite{cw}, does not carry a natural Lie bracket in general.

It is in examples arising from singular foliations that one sees the justification for the terminology ``pseudo-bundle'' in that the fibres of a pseudo-bundle of germs need not be isomorphic in general.

\begin{Example}\label{germfol}
Consider the foliation $\FF$ of $\RB$ generated by the vector field $X:=f\frac{\partial}{\partial x}$, where $f$ is any smooth function on $\RB$ such that $f(x) = 0$ for all $x\leq 0$ and such that $f(x)\neq 0$ for all $x>0$. Then for any $x_{0}<0$, the fibre $\Gamma_{\g}(\FF)_{x_{0}}$ consists of only a single point (every multiple of $X$ is equal to zero in a neighbourhood of $x_{0}$), while the fibre $\Gamma_{\g}(\FF)_{y_{0}}$ for any $y_{0}>0$ is the infinite-dimensional diffeological space consisting of all germs of smooth functions defined near the point $y_{0}$, and $\Gamma_{\g}(\FF)_{0}$ is the infinite-dimensional diffeological space consisting of germs of~smooth functions defined near zero that vanish at zero and below.
\end{Example}

\begin{Remark}
The pseudo-bundles of germs of singular foliations may be of particular interest~-- as we alluded to in Example~\ref{ex1}, they are canonically pseudo-bundles of diffeological Lie algebras. To~be more precise, let $\FF$ be a singular foliation of a manifold $M$. For each $x\in M$, let $\Gamma_{\loc}(\FF)_{x}$ be the diffeological subspace of $\Gamma_{\loc}(\FF)$ consisting of those local sections whose domains con\-tain~$x$, and let $q_{x}\colon \Gamma_{\loc}(\FF)_{x}\rightarrow\Gamma_{\g}(\FF)_{x}$ be the quotient map. Then it is routine to show that the formulae
\begin{gather*}
[q_{x}(X),q_{x}(Y)]:=q_{x}([X,Y]),\qquad
q_{x}(X)+q_{x}(Y) =q_{x}(X+Y),\qquad
\alpha q_{x}(X):=q_{x}(\alpha X)
\end{gather*}
for $X,Y\in\Gamma(\FF)_{x}$ and $\alpha\in\RB$ define a Lie algebra structure on $\Gamma_{\g}(\FF)_{x}$ which is smooth with respect to subspace diffeology from $\Gamma_{\g}(\FF)$. Thus $\pi_{\FF}^{\g}\colon \Gamma_{\g}(\FF)\rightarrow M$ is a diffeological vector pseudo-bundle of Lie algebras.
\end{Remark}

The final result of this subsection elucidates the nature of the tangent spaces to bundles of~germs. We~refer to Definition~\ref{tangentbundle} for notation.

\begin{Proposition}\label{tangentgerm}
Let $\pi_{B}\colon B\rightarrow M$ be a fibre bundle, with vertical tangent bundle $VB\rightarrow B$, and let $\big(x,[\sigma]^{\g}_{x}\big)\in\Gamma_{\g}(\pi_{B})$. Let $\Gamma_{\g}(\sigma^{*}VB)_{x}$ denote the vector space of germs at $x$ of sections of~$\sigma^{*}VB$. Then the map
\begin{gather}\label{tangerm}
\frac{\rm d}{{\rm d}t}\bigg|_{0}\big(\gamma(t),[\sigma_{t}]^{\g}_{\gamma(t)}\big) \mapsto \bigg(\frac{\rm d}{{\rm d}t}\bigg|_{0}\gamma(t),\bigg[\frac{\rm d}{{\rm d}t}\bigg|_{0}\sigma_{t}\bigg]_{x}\bigg),
\end{gather}
defines a linear isomorphism from $T_{\left(x,[\sigma]^{\g}_{x}\right)}\Gamma_{\g}(\pi_{B})$ to the vector space $T_{x}M\oplus\Gamma_{\g}(\sigma^{*}VB)_{x}$. In~particular, if ${\rm d}/{\rm d}t|_{0}\big(\gamma(t),[\sigma_{t}]^{\g}_{\gamma(t)}\big) = 0$, then there exists a neighbourhood $\OO$ of $x$ on which one has
\begin{gather*}
\frac{\rm d}{{\rm d}t}\bigg|_{0}\sigma_{t}(y) = 0
\end{gather*}
for all $y\in\OO$.
\end{Proposition}

\begin{proof}
First note that by definition of the diffeology on $\Gamma_{\g}(\pi_{B})$, any 1-plot $(-\epsilon,\epsilon)\rightarrow\Gamma_{\g}(\pi_{B})$ can, for sufficiently small $\epsilon$, be guaranteed to be of the form $t\mapsto \big(\gamma(t),[\sigma_{t}]^{\g}_{\gamma(t)}\big)$ for some plots $t\mapsto\sigma_{t}$ of $\Gamma_{\loc}(\pi_{B})$ and $t\mapsto\gamma(t)$ of $M$. Now observe that any sum of the form
\begin{gather*}
\frac{\rm d}{{\rm d}t}\bigg|_{0}\big(\gamma(t),[\sigma_{t}]^{\g}_{\gamma(t)}\big) +\frac{\rm d}{{\rm d}t}\bigg|_{0}\big(\tilde{\gamma}(t),[\tilde{\sigma}_{t}]^{\g}_{\tilde{\gamma}(t)}\big)
\end{gather*}
in $T_{\left(x,[\sigma]^{\g}_{x}\right)}\Gamma_{\g}(\pi_{B})$ can be represented by a single time derivative. By the first remark of this proof, to see this it suffices to work in $\Gamma_{\loc}(\pi_{B})$. Letting $\tilde{\OO}$ be a small neighbourhood of $\sigma(x)$ covering a small neighbourhood $\OO$ of $x$, let $\varphi\colon \tilde{\OO}\rightarrow\OO\times\RB^{m}$ denote the composite of coordinates on the fibre of $B$ with a local trivialisation of $B$ about $x$ such that $\varphi(\sigma(x)) = (x,0)$. Then the formula
\begin{gather*}
\psi(r,s):=\varphi^{-1}(\varphi\circ\sigma_{r}+\varphi\circ\tilde{\sigma}_{s})
\end{gather*}
defines a 2-parameter plot of $\Gamma_{\loc}(\pi_{B})$. Letting $\iota,\,\tilde{\iota}\colon \RB\rightarrow\RB^{2}$ denote the maps $t\mapsto(t,0)$ and $t\mapsto(0,t)$ respectively, we have $\sigma = \psi\circ\iota$ and $\tilde{\sigma} =\psi\circ\tilde{\iota}$, and therefore
\begin{gather*}
\frac{\rm d}{{\rm d}t}\bigg|_{0}\!\sigma_{t}+\frac{\rm d}{{\rm d}t}\bigg|_{0}\!\tilde{\sigma}_{t} =(\psi\circ\iota)_{*}\bigg(\frac{\rm d}{{\rm d}t}\bigg)+(\psi\circ\tilde{\iota})_{*}\bigg(\frac{\rm d}{{\rm d}t}\bigg)\! = \psi_{*}\bigg(\iota_{*}\bigg(\frac{\rm d}{{\rm d}t}\bigg)+\tilde{\iota}_{*}\bigg(\frac{\rm d}{{\rm d}t}\bigg)\!\bigg) = \psi_{*}\bigg(\frac{\rm d}{{\rm d}r},\frac{\rm d}{{\rm d}s}\bigg)
\end{gather*}
in $T_{\sigma}\Gamma_{\g}(\pi_{B})$. Now observe that if $h\colon \RB\rightarrow\RB^{2}$ denotes the map $t\mapsto(t,t)$, then, defining $\kappa_{t}:=\psi(t,t)$, we have
\begin{gather*}
\frac{\rm d}{{\rm d}t}\bigg|_{0}\kappa_{t} = (\psi\circ h)_{*}\bigg(\frac{\rm d}{{\rm d}t}\bigg) = \psi_{*}h_{*}\bigg(\frac{\rm d}{{\rm d}t}\bigg) = \psi_{*}\bigg(\frac{\rm d}{{\rm d}r},\frac{\rm d}{{\rm d}s}\bigg) = \frac{\rm d}{{\rm d}t}\bigg|_{0}\sigma_{t}+\frac{\rm d}{{\rm d}t}\bigg|_{0}\tilde{\sigma}_{t}.
\end{gather*}
Higher sums can be dealt with via a similar argument. Thus equation~\eqref{tangerm} does indeed suffice to define a map $T_{\left(x,[\sigma]^{\g}_{x}\right)}\Gamma_{\g}(\pi_{B})\rightarrow T_{x}M\oplus\Gamma_{\g}(\sigma^{*}VB)_{x}$, whose linearity is clear, and which admits the obvious inverse
\begin{gather*}
\bigg(\frac{\rm d}{{\rm d}t}\bigg|_{0}\gamma(t),\bigg[\frac{\rm d}{{\rm d}t}\bigg|_{0}\sigma_{t}\bigg]^{\g}_{x}\bigg)\mapsto \frac{\rm d}{{\rm d}t}\bigg|_{0}\big(\gamma(t),[\sigma]^{\g}_{\gamma(t)}\big).\tag*{\qed}
\end{gather*}
\renewcommand{\qed}{}
\end{proof}

\subsection{Relationship with sheaves}\label{relsheaf}

In this subsection we present some results and examples which relate our pseudo-bundles of germs to more well-known objects arising in sheaf theory. The~first such result, which is required to define the correct notion of morphism between singularly foliated bundles, is that a smooth morphism of sheaves gives rise to a morphism of the associated pseudo-bundles.

\begin{Proposition}\label{relsheaf1}
Let $M$ be a manifold and let $\SS^{1}$ and $\SS^{2}$ be sheaves of sections of fibre bundles $\pi_{B_{1}}$ and $\pi_{B_{2}}$ over $M$ respectively. Suppose that $\tilde{F}\colon \SS^{1}\rightarrow\SS^{2}$ is a morphism of sheaves for which the induced morphism $\SS^{1}_{\loc}\rightarrow\SS^{2}_{\loc}$ is smooth $($see Proposition~$\ref{functprop})$. Then the formula
\begin{gather*}
F\big(x,[\sigma]^{\g}_{x}\big):=\big(x,[\tilde{F}(\sigma)]^{\g}_{x}\big),\qquad \big(x,[\sigma]^{\g}_{x}\big)\in\Gamma_{\g}\big(\SS^{1}\big)
\end{gather*}
defines a morphism $F\colon \Gamma_{\g}\big(\SS^{1}\big)\rightarrow\Gamma_{\g}\big(\SS^{2}\big)$ of diffeological pseudo-bundles.
\end{Proposition}

\begin{proof}
It is clear that $F$ preserves fibres, so we need only check smoothness. Since for each $i=1,2$ the quotient diffeology on $\Gamma_{\g}\big(\SS^{i}\big)$ is inherited from the functional diffeology on $\SS^{i}_{\loc}$, smoothness of the map $\SS^{1}_{\loc}\rightarrow\SS^{2}_{\loc}$ associated to $\tilde{F}$ ensures smoothness of $F$.
\end{proof}

The converse of Proposition~\ref{relsheaf1} is \textit{not true} in general~-- namely, a smooth morphism of pseudo-bundles of germs \textit{need not arise} from any morphism (smooth or otherwise) of the underlying sheaves. This can be seen in the simplest of examples.

\begin{Example}
Consider $M = \RB$, and $B=\RB\times\RB$ with $\pi_{B}\colon B\rightarrow M$ the projection onto the first factor. Consider the map $F\colon \Gamma_{\g}(\pi_{B})\rightarrow\Gamma_{\g}(\pi_{B})$ defined by
\begin{gather*}
F\big(x,[f]^{\g}_{x}\big):=\big(x,[m_{x}f]^{\g}_{x}\big),\qquad \big(x,[f]^{\g}_{x}\big)\in\Gamma_{\g}(\pi_{B}),
\end{gather*}
where $m_{x}f$ denotes the function $y\mapsto xf(y)$. Then $F$ is smooth~-- indeed, if $U$ is any open subset of $\RB^{n}$ and $\tilde{x}\colon U\rightarrow\RB$ and $\tilde{f}\colon U\rightarrow C^{\infty}(\RB,\RB)$ are any two plots, then for each $x\in\RB$, smoothness of
\begin{gather*}
(u,y)\mapsto x\tilde{f}(u)(y)
\end{gather*}
guarantees that $F\circ\big(\tilde{x},[\tilde{f}]^{\g}_{\tilde{x}}\big)\colon \RB^{n}\rightarrow\Gamma_{\g}(\pi_{B})$ is smooth. Now suppose that $\tilde{F}\colon C^{\infty}_{\RB}\rightarrow C^{\infty}_{\RB}$ is a~morphism of sheaves. Then for $F$ to be induced by $\tilde{F}$, we must in particular have
\begin{gather*}
\big[\tilde{F}(\id)\big]^{\g}_{0} = [0\id]^{\g}_{0} = 0,
\end{gather*}
so that there must exist $\epsilon>0$ for which $\tilde{F}(\id)$ vanishes identically on $(-\epsilon,\epsilon)$. However, for $x\in(-\epsilon,\epsilon)\setminus 0$, we have
\begin{gather*}
\big[\tilde{F}(\id)\big]^{\g}_{x} = 0\neq[m_{x}\id]^{\g}_{x}.
\end{gather*}
Thus $F$ cannot arise from any morphism of sheaves.
\end{Example}

Morphisms of pseudo-bundles of germs which arise from morphisms of sheaves in the sense of~Pro\-position~\ref{relsheaf1} will play an important role in the correct notion of morphism between singularly foliated bundles. We~thus record the following definition.

\begin{Definition}
Let $M$ be a manifold and let $\SS_{1}$ and $\SS_{2}$ be sheaves of sections of fibre bundles over $M$. We~say that a morphism $F\colon \Gamma_{\g}(\SS_{1})\rightarrow\Gamma_{\g}(\SS_{2})$ is \textit{sheaf-induced} if it arises from a smooth morphism of the sheaves $\SS_{1}\rightarrow\SS_{2}$ in the sense of Proposition~\ref{relsheaf1}.
\end{Definition}

\begin{Remark}\label{relsheaf4}
For a sheaf $\SS$ of sections of a fibre bundle $\pi_{B}$ over a manifold $M$, we clearly have that $\Gamma_{\g}(\SS)$ is equal \textit{as a set} to the \'{e}tale space~\cite[p.~67]{hartshorne}
\begin{gather*}
E(\SS):=\bigsqcup_{x\in M}\SS_{x} = \Gamma_{\g}(\SS)
\end{gather*}
of the sheaf $\SS$. The~\'{e}tale space $E(\SS)$ is usually equipped with the \textit{\'{e}tale topology}, whose topology is generated by those sets of the form
\begin{gather*}
\UU(\sigma,\OO):=\big\{[\sigma]^{\g}_{x}\colon x\in\OO\big\}
\end{gather*}
defined for open sets $\OO$ of $M$ and $\sigma\in\SS(\OO)$. The~set $\Gamma_{\g}(\SS)$ may also be thought of with the D-topology (see Definition~\ref{Dtop}) arising from the diffeology on $\Gamma_{\g}(\SS)$ described in Definition~\ref{funl}, whose open sets are precisely those subsets $A$ for which $\rho^{-1}(A)$ is open in $\dom(\rho)$ for all plots~$\rho$ of~$\Gamma_{\g}(\SS)$.

It is easy to see that the D-topology is contained in the \'{e}tale topology. Suppose that a~non\-empty subset $A$ of $\Gamma_{\g}(\SS)$ is open in the D-topology, and fix a point $\big(x,[\sigma]_{x}^{\g}\big)\in A$. Choose a~representative $\sigma$ of $[\sigma]_{x}^{\g}$. We~find an open neighbourhood $\OO$ of $x$ in $M$ such that $\UU(\sigma,\OO)$ is contained in $A$. Let $U\subset\RB^{n}$ be an open set associated to a local coordinate system $\varphi\colon U\rightarrow M$ with $x\in\range(\varphi)$, which we assume to be small enough that $\range(\varphi)\subset\dom(\sigma)$. Then the parametrisation $\rho\colon U\rightarrow\Gamma_{\g}(\SS)$ defined by
\begin{gather*}
\rho(u):=\big(\varphi(u),[\sigma]^{\g}_{\varphi(u)}\big)
\end{gather*}
is a plot, and therefore
\begin{gather*}
\rho^{-1}(A) = \big\{u\in U\colon \big(\varphi(u),[\sigma]^{\g}_{\varphi(u)}\big)\in A\big\}
\end{gather*}
is an open subset of $\RB^{n}$. Defining $\OO:=\varphi\big(\rho^{-1}(A)\big)$, we see then that $\UU(\sigma,\OO)\subset A$

Curiously, the converse \textit{does not hold} in general. That is, the D-topology is usually \textit{strictly coarser} than the \'{e}tale topology. Consider $M = \RB$ and $B=\RB\times\RB$ the trivial bundle with $\pi_{B}\colon B\rightarrow M$ the projection onto the first factor. Fix $x_{0}\in\RB$, and consider the plot $\rho\colon \RB\rightarrow\Gamma_{\g}(\pi_{B})$ defined by
\begin{gather*}
\rho(t):=\big(x_{0},[f_{t}]^{\g}_{x_{0}}\big),
\end{gather*}
where $f_{t}$ is the map $x\mapsto tx$. Taking $\Gamma_{\g}(\pi_{B})$ with its \'{e}tale topology, the set
\begin{gather*}
\UU(\id,\RB):=\big\{\big(x,[\id]^{\g}_{x}\big)\colon x\in\RB\big\}
\end{gather*}
is open in $\Gamma_{\g}(\pi_{B})$. However,
\begin{gather*}
\rho^{-1}(\UU(\id,\RB)) = \{t\in\RB\colon \rho(t)\in\UU(\id,\RB)\} = \{1\}\subset\RB
\end{gather*}
is not open. Therefore \'{e}tale-open sets in $\Gamma_{\g}(\pi_{B})$ need not be open in the D-topology.
\end{Remark}

\subsection{The leafwise path category}

In~\cite{mac3}, we introduced a diffeological version of the Moore path category for any regular folia\-tion~$(M,\FF)$. The~objects of this category are simply points in $M$, while the morphisms are smooth, leafwise paths which have \textit{sitting instants} in that they are constant in small neighbourhoods of their endpoints. Composition of morphisms in this category is simply concatenation of paths. In~\cite{VilGar1}, the authors introduce an analogous diffeological space for \textit{singular} foliations, however concatenation of paths in this space no longer defines a category. In~this section, we~intro\-duce a hybrid of these two approaches~-- a diffeological space of integral curves of vector fields defining a singular foliation, for which concatenation of paths defines an associative and smooth multiplication.

We begin by recalling the definition of the path category of a diffeological space from~\cite{mac3} (cf.~\cite{ptfunctor}).

\begin{Definition}
Let $\XX$ be a diffeological space. The~\textit{path category of $\XX$} is the diffeological sub\-space $\PP(\XX)$ of the diffeological product $C^{\infty}(\RB_{\geq0},\XX)\times\RB_{\geq0}$ consisting of pairs $(\gamma,d)$ for which there exist neighbourhoods of $0$ and of $[d,\infty)$ in $\RB_{\geq0}$ on which $\gamma$ is constant. Path categories are functorial~-- given any smooth map $f\colon \XX\rightarrow \YY$ of diffeological spaces, the formula
\begin{gather*}
\PP(f)(\gamma,d):=(f\circ\gamma,d)
\end{gather*}
defines a smooth functor $\PP(f)\colon \PP(\XX)\rightarrow \PP(\YY)$ of diffeological categories.
\end{Definition}

Given any diffeological space $\XX$, range and source maps $r$ and $s$ mapping $\PP(\XX)\rightarrow \XX$ are defined respectively by $(\gamma,d)\mapsto \gamma(d)$ and $(\gamma,d)\mapsto\gamma(0)$, and whenever $r(\gamma_{2},d_{2}) = s(\gamma_{1},d_{1})$, we~define the product $(\gamma_{1}\gamma_{2},d_{1}+d_{2})$ of $(\gamma_{1},d_{1})$ and $(\gamma_{2},d_{2})$ by the formula
\begin{gather*}
\gamma_{1}\gamma_{2}(t):=
\begin{cases}
\gamma_{2}(t),&\text{for}\quad 0\leq t\leq d_{2},
\\
\gamma_{1}(t-d_{2}),&\text{for}\quad d_{2}\leq t<\infty.
\end{cases}
\end{gather*}
This product, together with the range and source maps, are smooth, so that $\PP(\XX)$ is a dif\-feo\-lo\-gi\-cal category~\cite[Proposition~3.22]{mac3}. Moreover~\cite[Proposition~3.23]{mac3} there is a smooth involution $\iota\colon \PP(\XX)\ni(\gamma,d)\mapsto\big(\gamma^{-1},d\big)\rightarrow\PP(\XX)$ defined by the formula
\begin{gather*}
\gamma^{-1}(t):=
\begin{cases}
\gamma(d-t),&\text{if}\quad 0\leq t\leq d,
\\
\gamma(0),&\text{for}\quad t\geq d.
\end{cases}
\end{gather*}
Under favourable circumstances, which will be explicated in this section, the involution $\iota$ des\-cends to a genuine inversion on certain diffeological quotients of $\PP(\XX)$, giving such quotients the structures of diffeological groupoids.

Suppose in particular that $(M,\FF)$ is a singularly foliated manifold, and let $\Gamma_{\g}(\FF)$ be the pseudo-bundle of germs of $\FF$. By connectedness of $\RB_{\geq0}:=[0,\infty)$, any smooth map $\tilde{\gamma}\colon [0,\infty)\rightarrow\Gamma_{\g}(\FF)$ has the form
\begin{gather}\label{notation}
\tilde{\gamma}(t) = \big(\gamma(t),[X(t)]^{\g}_{\gamma(t)}\big),\qquad t\in[0,\infty),
\end{gather}
where $\gamma\colon [0,\infty)\rightarrow M$ is a smooth curve, and where $X\colon [0,\infty)\rightarrow\Gamma_{\loc}(\FF)$ is a smooth function. We~will implicitly use the notation of equation~\eqref{notation} in what follows.

\begin{Definition}\label{leafwisepaths}
Let $(M,\FF)$ be a singularly foliated manifold. We~define the \textit{leafwise} or \textit{$\FF$-path category} $\PP(\FF)$ to be the diffeological subspace of $\PP(\Gamma_{\g}(\FF))$ consisting of triples $\big(\gamma,[X]^{\g},d\big)$ for which $X\colon [0,\infty)\rightarrow\Gamma_{\loc}(\FF)$ satisfies
\begin{enumerate}\itemsep=0pt
\item[(1)] $\dom(X(t))$ is equal to a fixed open neighbourhood of $\gamma([0,\infty)) = \gamma([0,d])$ for all $t$,
\item[(2)] $X(t)(\gamma(t)) = \dot{\gamma}(t)$ for all $t\in[0,\infty)$, and
\item[(3)] $[X(0)]^{\g}_{\gamma(0)} = 0$ and $[X(t)]^{\g}_{\gamma(t)} = 0$ for all $t\geq d$.
\end{enumerate}
With range and source onto $M$ given by $r\big(\gamma,[X]^{\g},d\big):=\gamma(d)$ and $s\big(\gamma,[X]^{\g},d\big):=\gamma(0)$ respectively, by item 3, $\PP(\FF)$ inherits composition from $\PP(\Gamma_{\g}(\FF))$ to become a diffeological category with object space $M$.
\end{Definition}

Definition~\ref{leafwisepaths} is in practice the same as~\cite[Definition~3.1]{VilGar1} given by Garmendia--Villatoro. Note crucially that Definition~\ref{leafwisepaths} requires strictly more information than just the path in~$M$-re\-quiring in addition a time-dependent extension of the tangent field of $\gamma$ to an open neighbourhood of $\gamma$. This is so that flows of elements of $\PP(\FF)$ determine germs of diffeomorphisms defined in open neighbourhoods of their sources. Thus Definition~\ref{leafwisepaths} may be contrasted with the simpler~\cite[Definition~3.23]{mac3} for regular foliations, where such an extension is not explicitly required. In~the regular case, the tangent field along a leafwise path can always be \textit{canonically} extended to a tangent field in a (transverse) open neighbourhood of the path in any foliated chart. The~same is not true in the singular setting.

We end the section by defining what we mean by a holonomy groupoid in the diffeological context. The~definition we give here is a mild generalisation of~\cite[Definition~3.26]{mac3}.

\begin{Definition}
Let $\XX$ be a diffeological space, and $\pi_{\BB}\colon \BB\rightarrow \XX$ be a diffeological pseudo-bundle. Let $P$ be a diffeological category with object space $\XX$. A~smooth functor $T\colon P\rightarrow\Aut(\pi_{\BB})$ (see Definition~\ref{structuregpd}) is called a \textit{transport functor} if there exists a smooth \textit{lifting map}
\begin{gather*}
L\colon \ P\times_{s,\pi_{\BB}}\BB\rightarrow\PP(\BB)
\end{gather*}
such that $T$ can be written as the composite
\begin{gather*}
T(\gamma,d)(b) = r\circ L((\gamma,d),b),\qquad ((\gamma,d),b)\in P\times_{s,\pi_{\BB}}\BB.
\end{gather*}
Note that smoothness of $T$ follows from smoothness of $L$. If in particular $P = \PP(\FF)$ is the leafwise path category of some singularly foliated manifold $(M,\FF)$, we refer to $T$ as a \textit{leafwise transport functor}.
\end{Definition}

An important consequence of the existence of a transport functor is the existence of an~asso\-ciated groupoid called the holonomy groupoid. This can be seen by the arguments of~\cite[Pro\-position~3.27]{mac3}.

\begin{Definition}\label{holgpd}
Let $\XX$ be a diffeological space, $\pi_{\BB}\colon \BB\rightarrow \XX$ a diffeological pseudo-bundle, and~$P$~a~diffeological category with object space $\XX$. If $T\colon P\rightarrow\Aut(\pi_{\BB})$ is a transport functor, then the quotient of $P$ by the equivalence relation
\begin{gather*}
\gamma_{1}\sim\gamma_{2}\Leftrightarrow T(\gamma_{1})=T(\gamma_{2})
\end{gather*}
on its space of morphisms is a diffeological groupoid over $\XX$ called the \textit{holonomy groupoid associated to $T$}.
\end{Definition}

\section{Singularly foliated bundles and their holonomy groupoids}\label{sec4}

\subsection{Singularly foliated bundles}

Singular foliations are generalisations of regular foliations, and are naturally associated to Lie groupoids more generally. In~each of these special cases, one has a notion of fibre bundle which is compatible with the additional structure~-- in the case of a Lie groupoid action, the correct notion is that of an \textit{equivariant bundle}, while for a regular foliation the correct notion is that of a \textit{foliated bundle} in the sense of Kamber and Tondeur~\cite{folbund}. We~give in this section what appears to be the first definition of a fibre bundle compatible with a singular foliation, which simultaneously generalises equivariant and foliated bundles.

First, notice that projectable vector fields on a fibre bundle $\pi_{B}\colon B\rightarrow M$ over a manifold $M$ \textit{do not} generally form a sheaf of $C^{\infty}_{B}$-modules over $B$. Indeed, if $X$ is any projectable vector field and $f\in C^{\infty}(B)$ is any function which is non-constant along the fibres of $\pi_{B}$, then equation~\eqref{projectableeqn} will in general no longer hold for the vector field $fX$. Projectable vector fields are, however, closed under multiplication by functions of the form $f\circ\pi_{B}$, where $f\in C^{\infty}(M)$. Since $\pi_{B}$ is an open map, we can formulate the following definition, which will play a crucial role in our definition of singularly foliated bundle.

\begin{Definition}
Let $\pi_{B}\colon B\rightarrow M$ be a fibre bundle. Denote by $C^{\infty}_{\proj,B}$ the subsheaf
\begin{gather*}
C^{\infty}_{\proj,B}(\OO):=\big\{f\circ\pi_{B}\in C^{\infty}_{B}(\OO)\colon f\in C^{\infty}_{M}(\pi_{B}(\OO))\big\}
\end{gather*}
of $C^{\infty}_{B}$, which we call the \textit{sheaf of projectable functions}. We~denote by $\XF_{\proj,B}$ the sheaf of $C^{\infty}_{\proj,B}$-modules
\begin{gather*}
\XF_{\proj,B}(\OO)\!:=\!\big\{X{\in}\XF_{B}(\OO)\colon \text{there is $(\pi_{B})_{*}X{\in}\XF_{M}(\pi_{B}(\OO))$ with ${\rm d}\pi_{B}\!\circ X\! = \! (\pi_{B})_{*}(X)\!\circ\!\pi_{B}$}\big\},
\end{gather*}
which we call the \textit{sheaf of projectable vector fields}.
\end{Definition}

The pushforward of projectable vector fields can now be characterised in the following sheaf-theoretic fashion. Recall that for a map $f\colon X\rightarrow Y$ of topological spaces and a sheaf $\SS$ on $X$, we use $f_{!}\SS$ to denote the pushforward of $\SS$ on $Y$~\cite[p.~65]{hartshorne}.

\begin{Proposition}
Let $\pi_{B}\colon B\rightarrow M$ be a fibre bundle. The~pushforward of projectable vector fields induces a morphism $(\pi_{B})_{*}\colon (\pi_{B})_{!}\XF_{\proj,B}\rightarrow\XF_{M}$ of sheaves of $C^{\infty}_{M}$-modules that preserves the Lie bracket.
\end{Proposition}

\begin{proof}
Notice first that we have a canonical isomorphism $C^{\infty}_{M}\cong (\pi_{B})_{!}C^{\infty}_{\proj,B}$ of sheaves of rings, obtained simply by sending $f\in C^{\infty}_{M}(\OO)$ to $f\circ\pi_{B}\in C^{\infty}_{\proj,B}(\pi_{B}^{-1}(\OO))$ for each open set $\OO$ in $M$. In~this way the $C^{\infty}_{\proj,B}$-module structure of $\XF_{\proj,B}$ indeed defines a $C^{\infty}_{M}$-module structure on $(\pi_{B})_{!}\XF_{\proj,B}$. The~pushforward $(\pi_{B})_{*}$ of $X\in\XF_{\proj,B}(\pi_{B}^{-1}(\OO))$ to $(\pi_{B})_{*}(X)\in\XF_{M}(\OO)$ then clearly preserves the associated $C^{\infty}_{M}(\OO)$-module structure for each open set $\OO$, and it is well-known~\cite[Lemma~3.10]{natopdiffgeom} that it also preserves the Lie bracket of vector fields.
\end{proof}

Singularly foliated bundles are now defined by \textit{singular partial connections}, which are particularly well-behaved partially-defined right-inverses of the pushforward morphism.

\begin{Definition}\label{singfolbund}
A \textit{singularly foliated bundle} is a triple $(\pi_{B},\FF,\ell)$, where $\pi_{B}\colon B\rightarrow M$ is a fibre bundle, $\FF$ is a singular foliation of $M$, and $\ell\colon \FF\rightarrow(\pi_{B})_{!}\XF_{\proj,B}$ is a morphism of sheaves of $C^{\infty}_{M}$-modules which preserves the Lie bracket, and for which the following hold.
\begin{enumerate}\itemsep=0pt
\item The morphism $\ell$ is a \textit{partial right-inverse} to $(\pi_{B})_{*}$ in the sense that
\begin{gather*}
(\pi_{B})_{*}\circ\ell = \id_{\FF},
\end{gather*}
on the sheaf $\FF$. In~particular this implies that $\ell$ is injective.
\item The morphism $\ell$ is \textit{complete} in the sense that for any open set $\OO$ in $M$, any $X\in\FF(\OO)$ and any $x\in\OO$, if $\Fl^{X}(x)$ is defined on an interval $I\subset\RB$ then so too is $\Fl^{\ell(X)}(b)$ for any $b\in B_{x}$.
\item The morphism $\ell$ is \textit{smooth} in the sense that the induced morphism
\begin{gather*}
\Gamma_{\loc}(\FF)\rightarrow \Gamma_{\loc}\big((\pi_{B})_{!}\XF_{\proj,B}\big)
\end{gather*}
of diffeological spaces is smooth with respect to the diffeology of Proposition~\ref{functprop}.
\end{enumerate}
We refer to such a morphism $\ell$ as a \textit{singular partial connection}.
\end{Definition}

We will usually denote a singularly foliated bundle $(\pi_{B},\FF,\ell)$ by simply $\pi_{B}$, with $\FF$ and $\ell$ assumed unless otherwise stated. Before discussing some examples, let us mention that completeness of $\ell$ does \textit{not} automatically follow from $\ell$ being a partial right inverse to $(\pi_{B})_{*}$. Indeed, it is easy to verify that for any open set $\OO\subset M$ and for any $X\in\FF(\OO)$, we have the relationship
\begin{gather*}
\Fl^{X}_{t}(x) = \pi_{B}\big(\Fl^{\ell(X)}_{t}(b)\big),\qquad x\in\OO,\ b\in B_{x}
\end{gather*}
between the flows of $\Fl^{X}(x)$ and $\Fl^{\ell(X)}(b)$ wherever they are defined. In~particular, that $\ell$ is a partial right-inverse to $(\pi_{B})_{*}$ implies that for any $b\in B_{x}$, the domain of $\Fl^{\ell(X)}(b)$ is contained in the domain of $\Fl^{X}(x)$. The~converse, however, does \textit{not} follow without completeness of $\ell$, as is easily seen by considering the standard example of $B =\RB^{2}$, $M = \RB$, and with $\ell\colon \XF(M)\rightarrow\XF_{\proj,B}(B)$ defined by
\begin{gather*}
\ell\bigg(f\frac{\partial}{\partial x}\bigg)(x,y):=f(x)\bigg(\frac{\partial}{\partial x}+y^{2}\frac{\partial}{\partial y}\bigg)
\end{gather*}
for $f\in C^{\infty}(M)$ and $(x,y)\in B$.

\begin{Example}[trivial bundles]\label{triv}
If $(M,\FF)$ is any singularly foliated manifold and $Q$ is any other manifold, then the trivial bundle $\pi\colon M\times Q\rightarrow M$ is canonically a singularly foliated bundle. Indeed, with respect to the decomposition $T(M\times Q)\cong TM\times TQ$, one has the trivial lift $\ell\colon \XF_{M}\rightarrow\pi_{!}\XF_{\proj,M\times Q}$ defined by the formula
\begin{gather*}
\ell(X):=(\pi^{*}(X),0)\in\XF_{M\times Q}\big(\pi^{-1}(\OO)\big) = (\pi_{!}\XF_{M\times Q})(\OO)
\end{gather*}
for all open sets $\OO$ in $M$ and $X\in\XF_{M}(\OO)$. Restricting $\ell$ to the subsheaf $\FF$ of $\XF_{M}$ one obtains a singular partial connection, with both completeness and smoothness being trivial.
\end{Example}

\begin{Example}[regularly foliated bundles]\label{regex}
Suppose that $\pi_{B}\colon B\rightarrow M$ is a \textit{regularly} foliated bundle, in the sense of Kamber--Tondeur~\cite[Definition~2.1]{folbund}~-- that is, there exists an involutive subbundle $T\FF_{B}\subset TB$ which is projected fibrewise-injectively to a subbundle $T\FF_{M}$ of $TM$. Involutivity of $T\FF_{B}$ implies that both $T\FF_{B}$ and $T\FF_{M}$ integrate to regular foliations of $B$ and $M$ respectively. Now if $\OO$ is any open subset of $M$ and $X\in\FF_{M}(\OO)$, then one obtains $\ell(X)\in\XF_{\proj,B}\big(\pi_{B}^{-1}(\OO)\big)$ whose value at a point $b\in B$ is the unique vector in $T_{b}\FF_{B}$ that is mapped by ${\rm d}\pi_{B}$ to $X(\pi_{B}(b))$. Clearly then the resulting morphism $\ell\colon \FF_{M}\rightarrow(\pi_{B})_{!}\XF_{\proj,B}$ is a singular partial connection in the sense of Definition~\ref{singfolbund}. Completeness and smoothness can both be seen by choosing foliated coordinates, in which $\ell$ is simply given by a trivial lift as in~Example~\ref{triv}.

Conversely, suppose that $\FF$ is a regular foliation of a manifold $M$, with leaf dimension $p$, and that $\pi_{B}\colon B\rightarrow M$ is a fibre bundle with a singular partial connection $\ell\colon \FF\rightarrow(\pi_{B})_{!}\XF_{\proj,B}$. In~a foliated chart $\OO\cong\RB^{p}\times\RB^{q}$ of $M$, wherein $\FF(\OO)$ is the $C^{\infty}_{M}(\OO)$-span of vector fields $\{e_{1},\dots,e_{p}\}$ that form the standard frame field of $\RB^{p}$, injectivity of $\ell$ implies that the vector fields $\{\ell(e_{1}),\dots,\ell(e_{p})\}$ span a $p$-dimensional subspace of $T_{b}B$ at each point $b\in\pi_{B}^{-1}(\OO)$ which intersects the vertical tangent space at $b$ only through zero. One thus obtains a smooth $p$-dimensional distribution $T\FF_{B}$ in $B$, and involutivity of $\FF$ together with the fact that $\ell$ preserves the Lie bracket implies that $T\FF_{B}$ is involutive. Thus $\pi_{B}\colon B\rightarrow M$ is a foliated bundle in~the sense of Kamber--Tondeur.
\end{Example}

\begin{Example}[equivariant bundles]\label{equiex}
Let $\GG$ be a Lie groupoid with unit space $M$. Denote the Lie algebroid of $\GG$ by $A:=\ker({\rm d}s)|_{M}$, with anchor map ${\rm d}r\colon A\rightarrow TM$, and let $\FF$ denote the associated sheaf of vector fields of the form ${\rm d}r\circ\sigma$, where $\sigma$ is an element of the sheaf of sections~$\AA$ of the Lie algebroid $A$. Let $\pi_{B}\colon B\rightarrow M$ be a fibre bundle.

If $\GG$ acts on $B$, then denote by $B\rtimes\GG$ the associated action groupoid~\cite[Example 2.2]{aac1}, with range and source denoted $r_{B}$ and $s_{B}$ respectively. Then the Lie algebroid $A_{B}:=\ker({\rm d}s_{B})|_{B}$ associated with the action groupoid $B\rtimes\GG$ is isomorphic to $\pi_{B}^{*}(A)$. For any open subset $\OO$ of~$M$, the formula
\begin{gather*}
\tilde{\ell}({\rm d}r\circ\sigma):={\rm d}r_{B}\circ\pi_{B}^{*}(\sigma),\qquad \sigma\in\AA(\OO)
\end{gather*}
then defines a singular partial connection. Completeness and smoothness are consequences of~the fact that $\tilde{\ell}$ is defined in terms of a smooth action of the groupoid $\GG$.
\end{Example}

\subsection{The holonomy groupoids of singularly foliated bundles}

One of the key objects in our construction of the holonomy groupoids of a singularly foliated bundle are pseudo-bundles of \textit{invariant} germs/jets. To~define these pseudo-bundles, we need \textit{vertical} prolongations of projectable vector fields (cf.~\cite[Definition~1.15]{varbi}).

\begin{Definition}
Let $\pi_{B}\colon B\rightarrow M$ be a fibre bundle, let $X\in\XF_{\proj}(B)$, and let $k$ denote any of~the symbols $1,\dots,\infty,\g$. The~vector field $\vpf^{k}(X)$ on $\Gamma_{k}(\pi_{B})$ defined by
\begin{gather*}
\vpf^{k}(X)\big(x,[\sigma]^{k}_{x}\big):=\frac{\rm d}{{\rm d}t}\bigg|_{t=0}\Big(x,\big[\Fl^{X}_{t}\circ\sigma\circ\Fl^{(\pi_{B})_{*}(X)}_{-t}\big]_{x}^{k}\Big)
\end{gather*}
is called the \textit{vertical $k$-prolongation} of $X$. On $\Gamma_{0}(\pi_{B}) = B$, we set $\vpf^{0}(X):=0$.
\end{Definition}

For a projectable vector field $X$ on a fibre bundle $\pi_{B}\colon B\rightarrow M$ and $k\geq1$, the vertical $k$-prolongation $\vpf^{k}(X)$ is the image of $\pf^{k}(X)$ under the canonical $\big(\pi_{B}^{k-1,k}\big)^{*}\big(V\pi_{B}^{k-1}\big)$-valued contact form $\theta^{(k)}$~\cite[Chapter~6.3]{saunders} on $\Gamma_{k}(\pi_{B})$. In~local coordinates $(x^{i},f^{\alpha},f^{\alpha}_{i},\dots)$, the components of~$\theta^{(k)}$ are given by
\begin{gather*}
\big(\theta^{(k)}\big)^{\alpha}_{I} = {\rm d}f^{\alpha}_{I} - f^{\alpha}_{Ii}{\rm d}x^{i}
\end{gather*}
for each $|I|<k$. Thus it is easily checked (cf.~equation~\eqref{prolongation}) that in coordinates the vertical prolongation of $X = a^{i}\frac{\partial}{\partial x^{i}}+b^{\alpha}\frac{\partial}{\partial f^{\alpha}}$ is given by
\begin{gather*}
\vpf^{k}(X) = \sum_{|I|=0}^{k-1}D_{I}^{(k)}(b^{\alpha}-a^{i}f^{\alpha}_{i})\frac{\partial}{\partial f^{\alpha}_{I}}.
\end{gather*}
The invariant pseudo-bundles of a singularly foliated bundle are now defined as follows.

\begin{Definition}\label{FFinv}
Let $\pi_{B}\colon B\rightarrow M$ be a singularly foliated bundle, and let $k$ denote any of the symbols $0,1,\dots,\infty,\g$. An element $\big(x,[\sigma]^{k}_{x}\big)$ of $\pi_{B}$ is said to be \textit{$\FF$-invariant to order $k$} if
\begin{gather*}
\vpf^{k}\big(\ell(X)\big)\big(x,[\sigma]^{k}_{x}\big) = 0
\end{gather*}
for all $X\in\FF$ defined in a neighbourhood $x$. We~say that $\pi_{B}$ \textit{admits enough conservation laws to order $k$} if at each $x\in M$ there is $\big(x,[\sigma]^{k}_{x}\big)$ which is $\FF$-invariant. This being the case, we call the corresponding sub-pseudo-bundle $\Gamma_{k}(\pi_{B})^{\FF}\subset\Gamma_{k}(\pi_{B})$ consisting of invariant germs/jets the \textit{$k^{\rm th}$ order $\FF$-invariant pseudo-bundle} of $\pi_{B}$. We~denote by $\pi_{B}^{k,\FF}$ the restriction of $\pi_{B}^{k}$ to $\Gamma_{k}(\pi_{B})^{\FF}$.
\end{Definition}

Note that if a singularly foliated bundle admits enough conservation laws to order $k>0$, then it admits enough conservation laws to any $l\leq k$. Moreover, by definition, \textit{every} singularly foliated bundle admits enough conservation laws to order 0.

Recall now~\cite[Definition~2.10]{mac3} that if $\pi_{B}\colon B\rightarrow M$ is a \textit{regularly} foliated bundle, a locally-defined section $\sigma$ of~$\pi_{B}$ is said to be \textit{distinguished} if, about any point in its image, there exist foliated coordinates $(x_{\alpha},y_{\alpha},f_{\alpha})$, with $x_{\alpha}$ and $y_{\alpha}$ denoting the leafwise and transverse coordinates respectively in the base, and with $f_{\alpha}$ denoting coordinates in the fibre, with respect to which $\sigma = \sigma(y_{\alpha})$ is independent of~the leafwise coordinates. We~denote by $\DS_{\g}(\pi_{B})$ the diffeological bundle of~germs of~the sheaf of~distinguished sections, and by $\DS_{k}(\pi_{B})$ the bundle of~jets of~distinguished sections. The~next proposition says that the $\FF$-invariant pseudo-bundles of~Defi\-ni\-tion~\ref{FFinv} generalise the bundles of~distinguished sections appearing in the regular case, and that therefore our constructions recover those of~\cite{mac3} in the regular case. In~particular, since distinguished functions of~this sort furnish the (local) degree 0 characteristic cohomology classes of~regular foliations~\cite[Example 1]{bryantgriffiths1}, we feel justified in referring to the $\FF$-invariant germs/jets of~a singularly foliated bundle as its \textit{conservation laws}.

\begin{Proposition}
Let $\pi_{B}\colon B\rightarrow M$ be a regularly foliated bundle, and let $k$ denote any of the symbols $0,\dots,\infty,\g$. Then the diffeological subspace $\Gamma_{k}(\pi_{B})^{\FF}$ coincides with the space $\DS_{k}(\pi_{B})$ of classes of distinguished sections.
\end{Proposition}

\begin{proof}
In foliated coordinates $(x_{\alpha},y_{\alpha},f_{\alpha})$ for $B$, corresponding to leafwise, transverse and fibre coordinates respectively, any element $X\in\FF$ is given by some $C^{\infty}_{M}$-linear combination
\begin{gather*}
X = a^{i}\frac{\partial}{\partial x_{\alpha}^{i}},
\end{gather*}
while $\ell(X)$ (see Example~\ref{regex}) is given by
\begin{gather*}
\ell(X) = \pi_{B}^{*}(a^{i})\frac{\partial}{\partial x_{\alpha}^{i}}.
\end{gather*}
Thus, in our coordinates $(x_{\alpha},y_{\alpha},f_{\alpha})\in\RB^{p}\times\RB^{q}\times\RB^{k}$, we have simply
\begin{gather*}
\Fl^{\ell(X)}_{t}(x_{\alpha},y_{\alpha},f_{\alpha}) = \big(\Fl^{X}_{t}(x_{\alpha}),y_{\alpha},f_{\alpha}\big)
\end{gather*}
for small $t$. It follows immediately that for any smooth function $\sigma\colon \RB^{p}\times\RB^{q}\rightarrow\RB^{k}$, the curve
\begin{gather*}
t\mapsto\big(\Fl^{\ell(X)}_{t}\circ(\id\times\sigma)\circ\Fl^{X}_{-t}\big)(x_{\alpha},y_{\alpha}) = \big(x_{\alpha},y_{\alpha},\sigma\big(\Fl^{X}_{t}(x_{\alpha}),y_{\alpha}\big)\big)
\end{gather*}
is constant in $t$ \textit{for all $X\in\FF$} if and only if $\sigma$ is constant in the $x$ coordinate. Thus $\Gamma_{\g}(\pi_{B})^{\FF} = \DS_{\g}(\pi_{B})$. Since foliated coordinates can always be used to extend an invariant jet to an invariant local section, one has $\Gamma_{k}(\pi_{B})^{\FF} = \DS_{k}(\pi_{B})$ for $k\leq\infty$ also.
\end{proof}

Defining lifting maps and leafwise transport functors for a singularly foliated bundle is now a simple matter of putting our definitions together.

\begin{Theorem}\label{liftsmooth}
Let $\pi_{B}\colon B\rightarrow M$ be a singularly foliated bundle. Let $k$ denote any of the symbols $0,\dots,\infty,\g$, and suppose that $\pi_{B}$ admits enough conservation laws to order $k$. Then for each $\big(\gamma,[X]^{\g},d\big)\in\PP(\FF)$, $\gamma(0) = x$, and for each $\big(x,[\sigma]^{k}_{x}\big)\in\Gamma_{k}(\pi_{B})^{\FF}_{x}$, the map
\begin{gather}\label{ivpsol}
t\mapsto\big(\gamma(t),\big[\Fl^{\ell(X)}_{t,0}\circ\sigma\circ\Fl^{X}_{0,t}\big]^{k}_{\gamma(t)}\big)
\end{gather}
is the unique solution to the initial value problem
\begin{gather}\label{ivp}
\frac{\rm d}{{\rm d}s}\bigg|_{t}f\big(s;x,[\sigma]^{k}_{x}\big) = \pf^{k}\big(\ell(X(t))\big)\big(f\big(t;x,[\sigma]^{k}_{x}\big)\big),\qquad
f\big(0;x,[\sigma]^{k}_{x}\big) = \big(x,[\sigma]^{k}_{x}\big)
\end{gather}
in the diffeological space $\Gamma_{k}(\pi_{B})^{\FF}$ for which $\pi_{B}^{k,\FF}\circ f\big({-};x,[\sigma]^{k}_{x}\big) = \gamma(-)$. Moreover, the lifting map $L\big(\pi_{B}^{k,\FF}\big)\colon \PP(\FF)\times_{s,\pi_{B}^{k,\FF}}\Gamma_{k}(\pi_{B})^{\FF}\rightarrow \PP\big(\Gamma_{k}(\pi_{B})^{\FF}\big)$ defined by
\begin{gather}\label{Leqn}
L\big(\pi_{B}^{k,\FF}\big)\big(\gamma,[X]^{\g},d;x,[\sigma]^{k}_{x}\big)(t):=
\big(\gamma(t),\big[\Fl^{\ell(X)}_{t,0}\circ\sigma\circ\Fl^{X}_{0,t}\big]^{k}_{\gamma(t)}\big),\qquad t\in[0,\infty)
\end{gather}
is smooth.
\end{Theorem}

Note that the expression on the right hand side of equation~\eqref{Leqn} only makes sense by the completeness assumption on the singular partial connection. To~prove Theorem~\ref{liftsmooth}, we require the following lemma.

\begin{Lemma}\label{technicallemma}
Let $(M,\FF)$ be a singular foliation. Suppose that $U$ is an open set in Euclidean space and $X\colon \RB_{\geq0}\times U\rightarrow\Gamma_{\loc}(\FF)$ and $\gamma\colon \RB_{\geq0}\times U\rightarrow M$ are smooth maps for which
\begin{enumerate}\itemsep=0pt
\item[$(1)$] $\dom(X(t,u)) = \dom(X(s,u))$ for all $s,t\in\RB_{\geq0}$ and $u\in U$, and for which
\item[$(2)$] $t\mapsto\gamma(t,u)$ is an integral curve of $X(t,u)$ for all $(t,u)\in\RB_{\geq0}\times U$. That is,
\begin{gather*}
\frac{\rm d}{{\rm d}s}\bigg|_{s=t}\gamma(s,u) = X(t,u)(\gamma(t,u))
\end{gather*}
for all $(t,u)\in\RB_{\geq0}\times U$.
\end{enumerate}
Then for each $t_{0}\in\RB_{\geq0}$ and $u_{0}\in U$, there exist $\epsilon>0$ and open neighbourhoods $\UU\ni u_{0}$ in $U$ and $\OO\ni x_{0}:=\gamma(0,u_{0})$ in $M$, such that $\OO\subset\dom\big(\Fl^{X(u)}_{t,0}\big)$ for all $(t,u)\in[0,t_{0}+\epsilon)\times\UU$, and for which the map
\begin{gather*}
[0,t_{0}+\epsilon)\times\UU\times\OO\ni(t,u,x)\mapsto\Fl^{X(u)}_{t,0}(x)\in M
\end{gather*}
is smooth.
\end{Lemma}

\begin{proof}
By the definition of the diffeology on $\Gamma_{\loc}(\FF)$ and hypothesis 1, we can find open sets $\UU'\ni u_{0}$ and $\OO'\ni x_{0}$ such that $\OO'\subset\dom(X(t,u))$ for all $(t,u)\in\RB_{\geq0}\times\UU'$, and on which
\begin{gather*}
\RB_{\geq0}\times \UU'\times\OO'\ni(t,u,x)\mapsto X(t,u)(x)\in TM
\end{gather*}
is smooth. By hypothesis 2 we can assume that $\OO'$ contains $\gamma(\RB_{\geq0}\times\{u_{0}\})$. Now $X$ may be regarded as a time-dependent vector field on the manifold $\UU'\times\OO'$, and then by standard theory~\cite[Theorem~9.48]{leesmth}, there exists some maximal open neighbourhood $\NN\subset\UU'\times\OO'$ of~$(u_{0},x_{0})$ on which $t\mapsto\Fl^{X(u)}_{t,0}(x)$ is defined for all $u\in\UU$, $x\in\OO$ and for small $t$. Since by hypothesis $t\mapsto\Fl^{X(u_{0})}_{t,0}(x_{0}) = \gamma(t,u_{0})$ is defined for all $t\in\RB_{\geq0}$ we can always choose $\epsilon>0$, $\UU\ni u_{0}$ and $\OO\ni x_{0}$ small enough that $(t,u,x)\mapsto\Fl^{X(u)}_{t,0}(x)$ is well-defined and smooth on $[0,t_{0}+\epsilon)\times\UU\times\OO$ as claimed.
\end{proof}

\begin{proof}[Proof of Theorem~\ref{liftsmooth}]
That equation~\eqref{ivpsol} defines a curve in $\Gamma_{k}(\pi_{B})^{\FF}$ follows from the inva\-riance of $\FF$ under its own (possibly time-dependent~\cite[Lemma~3.3]{VilGar1}) flows~\cite[Proposition~1.6]{iakovos2}, the $\FF$-invariance of $\sigma$, and that $\ell$ is bracket preserving. More precisely, given $t\in[0,\infty)$, an open neighbourhood $\OO$ of $\gamma(t)$ and $Y\in\FF(\OO)$, since $\ell$ is bracket preserving we have
\begin{gather}\label{bracketpres}
\ell([X(t),Y])\big(\Fl^{\ell(X)}_{t,0}(b)\big) = [\ell(X(t)),\ell(Y)]\big(\Fl^{\ell(X)}_{t,0}(b)\big)
\end{gather}
for all $b\in\Fl^{\ell(X)}_{0,t}(\OO)$. Now, for any such $b$, with $\pi_{B}(b) = x'$, the right hand side of equation~\eqref{bracketpres} is the lift by $\ell$ of the tangent to the curve $r\mapsto\big(\Fl^{X}_{0,t+r}\big)_{*}\big(Y\big(\Fl_{t+r,0}^{X}(x')\big)\big)$ at $r = 0$,
while the left hand side is the tangent to the curve $r\mapsto \big(\Fl^{\ell(X)}_{0,t+r}\big)_{*}\big(\ell(Y)\big(\Fl^{\ell(X)}_{t+r,0}(b)\big)\big)$ at $r=0$. By uniqueness of flows therefore, we have
\begin{gather}\label{ellhom}
\big(\Fl^{\ell(X)}_{0,t}\big)_{*}(\ell(Y)) = \ell\big(\big(\Fl^{X}_{0,t}\big)_{*}(Y)\big).
\end{gather}
For notational simplicity denote $\varphi:=\Fl^{X}_{t,0}$ and $\ell(\varphi):=\Fl^{\ell(X)}_{t,0}$. Then we use equation~\eqref{ellhom} to compute
\begin{align*}
\vpf^{k}(\ell(Y))\big(\gamma(t),[\ell(\varphi)\!\circ\!\sigma\!\circ\!\varphi^{-1}]^{k}_{\gamma(t)}\big)
&= \frac{\rm d}{{\rm d}s}\bigg|_{s=0}\!\!\big(\gamma(t),\big[\Fl^{\ell(Y)}_{s}\circ\ell(\varphi) \circ\sigma\circ\varphi^{-1}\circ\Fl^{Y}_{-s}\big]^{k}_{\gamma(t)}\big)
\\
&= \frac{\rm d}{{\rm d}s}\bigg|_{s=0}\!\!\big(\gamma(t),\!\big[\ell(\varphi) \!\circ\Fl^{\ell(\varphi^{-1}_{*}(Y))}_{s}\! \circ\sigma\circ\Fl^{\varphi^{-1}_{*}(Y)}_{-s} \circ\varphi^{-1}\big]^{k}_{\gamma(t)}\big)
\\
&= 0
\end{align*}
by the $\FF$-invariance of $\sigma$ and since $\varphi^{-1}_{*}(Y)\in\FF$. Therefore equation~\eqref{ivpsol} does indeed define a curve in $\Gamma_{k}(\pi_{B})^{\FF}$. That equation~\eqref{ivpsol} defines a solution to the initial value problem of~equation~\eqref{ivp} follows by definition of the prolongation.

Uniqueness for $k$ finite follows from uniqueness of the flow in a manifold (since each $\Gamma_{k}(\pi_{B})^{\FF}$ is a subset of the jet manifold $\Gamma_{k}(\pi_{B})$), while for $k=\infty$ any solution to equation~\eqref{ivp} is the projective limit of the solutions for finite $k$, so uniqueness in this case follows from uniqueness for each finite $k$. Since $\Gamma_{\g}(\pi_{B})$ is neither a manifold nor a projective limit of manifolds, the argument is more subtle in the $k=\g$ case. Suppose that $\rho\colon [0,d]\rightarrow\Gamma_{\g}(\pi_{B})^{\FF}$ is some other solution to the initial value problem given in equation~\eqref{ivp}; thus in particular $\rho(0) = \big(x,[\sigma]^{\g}_{x}\big)$. Let us consider $\epsilon>0$ sufficiently small that for all $t\in[0,\epsilon)$ we can write $\rho(t) = \big(\gamma(t),[\sigma_{t}]^{\g}_{\gamma(t)}\big)$ for some $\FF$-invariant family $t\mapsto\sigma_{t}$ in $\Gamma_{\loc}(\pi_{B})$. By definition of the diffeology on $\Gamma_{\loc}(\pi_{B})$ we may assume that $\epsilon$ is sufficiently small that there is some open neighbourhood $\OO$ of $\gamma(0)$ containing $\gamma([0,\epsilon))$ and on which $\sigma_{t}$ is defined for all $t\in[0,\epsilon)$. Let us also assume that there are fibre bundle coordinates $\big(x^{1},\dots,x^{n};b^{1},\dots,b^{m}\big)$ on $B$ about $\sigma(x)$, the projection of whose domain to~$M$ contains $\OO$. By hypothesis we have\vspace{-1ex}
\begin{gather*}
\frac{\rm d}{{\rm d}s}\bigg|_{t}[\sigma_{s}]^{\g}_{\gamma(s)} = \frac{\rm d}{{\rm d}s}\bigg|_{t}\big[\Fl^{\ell(X)}_{s,0}\circ\sigma\circ\Fl^{X}_{0,s}\big]_{\gamma(s)}
\end{gather*}
for all $t\in[0,\epsilon)$, which by Proposition~\ref{tangentgerm} implies the coordinate expression\vspace{-1ex}
\begin{gather*}
\frac{\rm d}{{\rm d}s}\bigg|_{t}\big(\sigma^{i}_{s}(x^{1},\dots,x^{n})- \big(\Fl^{\ell(X)}_{s,0}\circ\sigma\circ\Fl^{X}_{0,s}\big)^{i}\big(x^{1},\dots,x^{n}\big)\big) = 0.
\end{gather*}
for $i=1,\dots,m$ and for $\big(x^{1},\dots,x^{n}\big)\in\OO$. This being true for all $t\in[0,\epsilon)$, it follows that the difference $\sigma^{i}_{t}-\big(\Fl^{\ell(X)}_{t,0}\circ\sigma\circ\Fl^{X}_{0,t}\big)^{i}$ is constant in $t$ on $\OO$ for all $i$, and since $\sigma^{i}_{0} = \sigma^{i}$ on $\OO$ we obtain $[\sigma_{t}]^{\g}_{\gamma(t)} = \big[\Fl^{\ell(X)}_{t,0}\circ\sigma\circ\Fl^{X}_{0,t}\big]^{\g}_{\gamma(t)}$ for all $t\in[0,\epsilon)$. Uniqueness on $[0,d]$ now follows from the compactness of $[0,d]$.

It remains only to show smoothness of $L\big(\pi_{B}^{k,\FF}\big)$. Let $\varphi_{\PP}:=\big(\tilde{\gamma},[\tilde{X}]^{\g},\tilde{d}\big)\colon U\rightarrow\PP(\FF)$ and $\varphi_{\Gamma}:=\big(\tilde{x},[\tilde{\sigma}]^{k}_{\tilde{x}}\big)\colon V\rightarrow\Gamma_{k}(\pi_{B})^{\FF}$ be plots. We~need to show that the map from $W:=U\times_{s\circ\varphi_{\PP},\pi_{B}^{k,\FF}\circ\varphi_{\Gamma}}V\times\RB_{\geq0}$ to $\Gamma_{k}(\pi_{B})^{\FF}$ defined by
\begin{gather*}
(u,v,t)\mapsto \big(\tilde{\gamma}(u)(t),\big[\Fl^{\ell(\tilde{X}(u))}_{t,0}\circ \tilde{\sigma}(v)\circ\Fl^{\tilde{X}(u)}_{0,t}\big]^{k}_{\tilde{\gamma}(u)(t)}\big)
\end{gather*}
is smooth. The~map $\tilde{\gamma}$ is already a plot of $\PP(M)$ by definition. Recalling that $(\Gamma_{\pi_{B}})_{\loc}$ denotes the space of all locally defined sections of $\pi_{B}$ equipped with the diffeology of Proposition~\ref{functprop}, it~suffices to show that the map
\begin{gather*}
(u,v,t)\mapsto\kappa(u,v,t):=\big(\Fl^{\ell(\tilde{X}(u))}_{t,0}\circ\tilde{\sigma}(v)\circ\Fl^{\tilde{X}(u)}_{0,t}\big)\in(\Gamma_{\pi_{B}})_{\loc}
\end{gather*}
is smooth. That is, fixing $(u_{0},v_{0},t_{0})\in W$ and $x_{0}\in\dom(\kappa(u_{0},v_{0},t_{0}))$, we must find an open neighbourhood $\WW$ of $(u_{0},v_{0},t_{0})$ in $W$ and an open neighbourhood $\OO$ of $x_{0}$ in $M$ such that $\OO\subset\dom(\kappa(u,v,t))$ for all $(u,v,t)\in\WW$ and for which the map
{\samepage\begin{gather*}
\WW\times\OO\ni(u,v,t,x)\mapsto \kappa(u,v,t)(x)\in B
\end{gather*}
is smooth in the usual sense. Setting $d_{0}:=d(u_{0})$, we have the following.

}

\begin{enumerate}\itemsep=0pt
\item Using Lemma~\ref{technicallemma} together with the smoothness of $\ell$ as a map $\Gamma_{\loc}(\FF)\rightarrow(\pi_{B})_{*}(\XF_{B})_{\loc}$, we can find open neighbourhoods $\OO^{B}\ni\tilde{\sigma}(v_{0})\big(\Fl^{\tilde{X}(u_{0})}_{0,t_{0}}(x_{0})\big)$ in $B$, $\UU_{1}\ni u_{0}$ in $U$ and $I_{1}\supset[0,t_{0}]$ in $\RB_{\geq0}$ for which $\OO^{B}\subset\dom\big(\Fl^{\ell(\tilde{X}(u))}_{t,0}\big)$ for all $(u,t)\in\UU_{1}\times I_{1}$, and such that $I_{1}\times\UU_{1}\times \OO^{B}\ni(t,u,b)\mapsto\Fl^{\ell(\tilde{X}(u))}_{t,0}(b)\in B$ is smooth.
\item By definition of the diffeology on $(\Gamma_{\pi_{B}})_{\loc}$, we can find open neighbourhoods $\OO^{M}\ni\Fl^{\tilde{X}(u_{0})}_{0,t_{0}}(x_{0})$ in $M$ and $\VV\ni v_{0}$ in $V$ such that $\OO^{M}\subset\dom(\tilde{\sigma}(v))$ and $\tilde{\sigma}(v)(\OO^{M})\subset\OO^{B}$ for all $v\in\VV$, and for which $\VV\times\OO^{M}\ni(v,x)\mapsto\tilde{\sigma}(v)(x)\in B$ is smooth.
\item Again by Lemma~\ref{technicallemma}, we can find open neighbourhoods $\UU_{2}\ni u_{0}$, $I_{2}\supset[0,t_{0}]$ in $\RB_{\geq0}$ and $\OO\ni x_{0}$ in $M$ such that $\OO\subset\dom\big(\Fl^{\tilde{X}(u)}_{0,t}\big)$ and $\Fl^{\tilde{X}(u)}_{0,t}(\OO)\subset\OO^{M}$ for all $u\in\UU_{2}$ and $t\in I_{2}$, and such that $I_{2}\times\UU_{2}\times \OO\ni(t,u,x)\mapsto\Fl^{\tilde{X}(u)}_{0,t}(x)\in M$ is smooth.
\end{enumerate}
Finally, therefore, setting $\UU:=\UU_{1}\cap\UU_{2}$, $I:=I_{1}\cap I_{2}$ and
\begin{gather*}
\WW:=\UU\times_{s\circ\varphi_{\PP},\pi_{B}^{k,\FF}\circ\varphi_{\Gamma}}\VV\times I\subset W,
\end{gather*}
we have $\OO\subset\dom(\kappa(u,v,t))$ for all $(u,v,t)\in\WW$ and $\WW\times\OO\ni(u,v,t,x)\mapsto\kappa(u,v,t)(x)\in B$ is smooth.
\end{proof}

\begin{Definition}\label{transportfunctor}
Let $\pi_{B}\colon B\rightarrow M$ be a singularly foliated bundle. Let $k$ denote any of the symbols $0,\dots,\infty,\g$, and suppose that $\pi_{B}$ admits enough conservation laws to order $k$. Then the map $T\big(\pi_{B}^{k,\FF}\big)\colon \PP(\FF)\rightarrow\Aut\big(\pi_{B}^{k,\FF}\big)$ defined by
\begin{gather*}
T\big(\pi_{B}^{k,\FF}\big)\big(\gamma,[X]^{\g},d\big)\big(x,[\sigma]^{k}_{x}\big):= L\big(\pi_{B}^{k,\FF}\big)\big(\gamma,[X],d;x,[\sigma]^{k}_{x}\big)(d)
\end{gather*}
is a transport functor called the \textit{$k$-transport functor} for $\pi_{B}$. The~associated holonomy groupoid (see Definition~\ref{holgpd}) is called the \textit{$k$-holonomy groupoid} of $\pi_{B}$ and denoted $\HH\big(\pi_{B}^{k,\FF}\big)$.
\end{Definition}

Note that since each holonomy groupoid of Definition~\ref{transportfunctor} arises as a quotient of the leafwise path category, \textit{every one of them} integrates the foliation of the base in the sense of Theorem~\ref{integrate}. Finally, we have the following analogue of~\cite[Theorem~5.15]{mac3} which relates all of the holonomy groupoids of a singularly foliated bundle.

In what follows we will assume the set $\{0,1,\dots,\infty,\g\}$ to be equipped with the total order which coincides with the usual one on $\NB$, and for which $k<\infty<\g$ for all $k\in\NB$.

\begin{Theorem}\label{hierarchy}
Let $\pi_{B}\colon B\rightarrow M$ be a singularly foliated bundle, and suppose that $\pi_{B}$ admits enough conservation laws to order $k$. Then for each $l\leq k$, there is a subductive groupoid morphism $\Pi_{B}^{l,k}\colon \HH\big(\pi_{B}^{k,\FF}\big)\rightarrow\HH(\pi_{B}^{l,\FF})$ such that $\Pi_{B}^{m,k} = \Pi_{B}^{m,l}\circ\Pi_{B}^{l,k}$ for all $m\leq l\leq k$. In~particular, if $\pi_{B}$ admits enough conservation laws to order $\g$, we have a commuting diagram
\begin{center}
\begin{tikzcd}[row sep = large]
& & & \HH\big(\pi_{B}^{\g,\FF}\big) \ar[d,"\Pi_{B}^{\infty,\g}"] & & &
\\
& & & \HH\big(\pi_{B}^{\infty,\FF}\big) \ar[dl,"\Pi_{B}^{k+1,\infty}"] \ar[dr,"\Pi_{B}^{k,\infty}"'] \ar[drrr,"\Pi_{B}^{0,\infty}"'] & & &
\\
& \cdots\ar[r] & \HH\big(\pi_{B}^{k+1,\FF}\big) \ar[rr,"\Pi_{B}^{k,k+1}"'] & & \HH\big(\pi_{B}^{k,\FF}\big) \ar[r] & \cdots \ar[r] & \HH\big(\pi_{B}^{0,\FF}\big)
\end{tikzcd}
\end{center}
of diffeological groupoids, which we refer to as the \textit{hierarchy of holonomy groupoids} for the singularly foliated bundle $\pi_{B}$.
\end{Theorem}

\begin{proof}
The proof is similar to that of~\cite[Theorem~5.15]{mac3}.
\end{proof}

\subsection{Agreement with the Garmendia--Villatoro construction}

We show in this subsection that for trivial singularly foliated bundles, the germinal holonomy groupoid of Definition~\ref{transportfunctor} (all trivial singularly foliated bundles admit enough conservation laws to order $\g$~-- one need only take the constant functions) is a quotient of the holonomy groupoid constructed by Garmendia--Villatoro in~\cite{VilGar1}, hence~\cite[Theorem~5.5]{VilGar1} is also a quotient of the Androulidakis--Skandalis holonomy groupoid~\cite{iakovos2}. A~key feature of the Garmendia--Villatoro con\-st\-ruction is the use of \textit{slices}.

\begin{Definition}
Let $(M,\FF)$ be a singularly foliated manifold. A~\textit{slice} through a point $x\in M$ is an embedded submanifold $S_{x}\hookrightarrow M$ such that $T_{x}S\cap T_{x}L_{x} = 0$, and such that $T_{y}M = T_{y}S+T_{y}L_{y}$ for all $y\in S$.
\end{Definition}

Given a singularly foliated manifold $(M,\FF)$, Garmendia and Villatoro attach to each point $x\in M$ a slice $S_{x}$, and denote by $\g\Diff_{\FF}(S_{x},S_{y})$ the set of germs of foliation-preserving diffeomorphisms from $S_{x}$ to $S_{y}$. Letting $I_{x}$ denote the ideal of smooth functions vanishing at $x$, the group $\g\Diff_{\FF}(S_{x},S_{x})$ admits a subgroup $\exp(I_{x}\FF|_{S_{x}})$ consisting of flows of (possibly time-dependent) elements of $I_{x}\FF|_{S_{x}}$. Garmendia and Villatoro then define the groupoid
\begin{gather*}
\text{HT}:=\bigsqcup_{x,y\in M}\g\Diff_{\FF}(S_{x},S_{y})/\exp\big(I_{y}\FF|_{S_{y}}\big)
\end{gather*}
of holonomy transformations. Any element $(\gamma,[X]^{\g},d)\in\PP(\FF)$ defines an element $\text{Hol}(\gamma,[X]^{\g},d)$ of $\text{HT}$ by choosing~\cite[Lemma~A.8]{iakovos4} $Z\in I_{\gamma(d)}\FF$ such that $\Fl^{Z}_{1}\circ\Fl^{X}_{d}$ maps a neighbourhood of~$\gamma(0)$ in $S_{\gamma(0)}$ onto a neighbourhood of $\gamma(d)$ in $S_{\gamma(d)}$. Then the class of
\begin{gather*}
\text{Hol}(\gamma,[X]^{\g},d):=\big[\Fl^{Z}_{1,0}\circ\Fl^{X}_{d,0}\big]^{\g}_{\gamma(0)}
\end{gather*}
in HT is independent of the choice of $Z$, and one thereby obtains a map $\text{Hol}\colon \PP(\FF)\rightarrow {\rm HT}$. The~Garmendia--Villatoro holonomy groupoid is now the diffeological quotient of $\PP(\FF)$ by the fibres of $\text{Hol}$.

\begin{Theorem}\label{agree}
Let $(M,\FF)$ be a singularly foliated manifold of dimension $n$. Then the holonomy groupoid $\HH\big(\pi_{M\times\RB^{n}}^{\g,\FF}\big)$ associated to the trivial singularly foliated bundle \mbox{$\pi_{M\times\RB^{n}}\colon M\times\RB^{n}\rightarrow M$} $($see Example~$\ref{triv})$ is the quotient of the Garmendia--Villatoro holonomy groupoid by the equivalence relation which identifies groupoid elements if and only if they induce the same parallel transport map on $n$-tuples of first integrals. In~particular, if $\FF$ is regular of codimension $q\leq n$, then $\HH\big(\pi_{M\times\RB^{n}}^{\g,\FF}\big)$ coincides with the Garmendia--Villatoro holonomy groupoid.
\end{Theorem}

In order to prove Theorem~\ref{agree}, we need to show that if two elements of the leafwise path space $\PP(\FF)$ of such a foliation are mapped to the same germ under $\text{Hol}$ then they are mapped to the same diffeomorphism in $\Aut\big(\pi_{M\times\RB^{n}}^{\g,\FF}\big)$ under the transport functor $T$ of Definition~\ref{transportfunctor}, and conversely for $\FF$ regular. To~show that this is true let us discuss the relationship between slices in $(M,\FF)$ and $\FF$-invariant local sections of $M\times\RB^{n}$.

Slices are always found inside certain foliated charts. We~recall~\cite[Proposition~1.3]{iakovos4} that if the dimension of the leaf $L_{x}$ through $x$ is $p$, and if $S_{x}$ is a slice through $x$, then there exists an open neighbourhood $\OO$ of $x$ in $M$ and a diffeomorphism of foliated manifolds
\begin{gather}\label{chart}
\big(\OO,\FF|_{\OO}\big)\cong\big(\RB^{p},\XF_{\RB^{p}}\big)\times(S_{x},\FF|_{S_{x}}).
\end{gather}
In these coordinates, every $\FF$-invariant section of $M\times\RB^{n}\rightarrow M$ takes the form
\begin{gather}\label{sigmaf}
\sigma(a,b) = (a,b,f(b)),\qquad (a,b)\in\RB^{p}\times S_{x},
\end{gather}
where $f\colon S_{x}\rightarrow\RB^{n}$ is an $\FF$-invariant function.

\begin{proof}[Proof of Theorem~\ref{agree}]
Denote the transport functor $T\big(\pi_{M\times\RB^{n}}^{\g,\FF}\big)$ by simply $T$. Given ele\-ments $\big(\gamma_{i},[X_{i}]^{\g},d_{i}\big)\in\PP(\FF)$, $i = 1,2$, for item 1 we must show that $\text{Hol}\big(\gamma_{1},[X_{1}]^{\g},d_{1}\big) = \text{Hol}\big(\gamma_{2},[X_{2}]^{\g},d_{2}\big)$ implies $T\big(\gamma_{1},[X_{1}]^{\g},d_{1}\big) = T\big(\gamma_{2},[X_{2}]^{\g},d_{2}\big)$, and conversely for item 2 subject to the additional hypothesis. Clearly in either case $\gamma_{1}$ and $\gamma_{2}$ must have the same source and range, which we denote by $x$ and $y$ respectively, and each $X_{i}$ defines a diffeomorphism $\varphi_{i}:=\Fl^{X_{i}}_{d_{i},0}$ of~some open neighbourhood $\OO_{x}$ of $x$ onto an open neighbourhood $\OO_{y}$ of $y$. We~may assume that $\OO_{x}$ and $\OO_{y}$ are of the form given in equation~\eqref{chart} for the slices $S_{x}$ and $S_{y}$ about~$x$ and~$y$ respectively that are used to define HT. Let us fix $Z_{i}$, $i=1,2$, such that $\Fl^{Z_{i}}_{1,0}\circ\Fl^{X_{i}}_{d_{i},0}$ maps a~neighbourhood of $x$ in $S_{x}$ onto a neighbourhood of~$y$ in $S_{y}$.

First suppose that $\text{Hol}\big(\gamma_{1},[X_{1}]^{\g},d_{1}\big) = \text{Hol}\big(\gamma_{2},[X_{2}]^{\g},d_{2}\big)$. Then there exists a possibly time-dependent element $W$ of $I_{x}\FF|_{S_{x}}$ such that
\begin{gather*}
\big[\Fl^{Z_{1}}_{1,0}\circ\Fl^{X_{1}}_{d_{1},0}\big|_{S_{x}}\big]^{\g}_{x} = \big[\Fl^{W}_{1,0}\circ\Fl^{Z_{2}}_{1,0}\circ\Fl^{X_{2}}_{d_{2},0}\big|_{S_{x}}\big]^{\g}_{x}
\end{gather*}
as germs of maps $S_{x}\rightarrow S_{y}$ at $x$. Now for any $\FF$-invariant section $\sigma$ defined in an open neighbourhood $y$ of $M$, we have
\begin{gather*}
\Fl^{\ell(X_{1})}_{0,d_{1}}\circ\sigma\circ\Fl^{X_{1}}_{d_{1},0} = \Fl^{\ell(X_{1})}_{0,d_{1}}\circ\Fl^{\ell(Z_{1})}_{0,1}\circ\sigma\circ\Fl^{Z_{1}}_{1,0}\circ\Fl^{X_{1}}_{d_{1},0}
\end{gather*}
on an open neighbourhood of $x$. Choosing coordinates about $x$ and $y$ defined respectively by~$S_{x}$ and $S_{y}$ as in equation~\eqref{chart}, with $\sigma$ given by an $\FF$-invariant function $f\colon S_{y}\rightarrow\RB^{n}$ as in equation~\eqref{sigmaf}, we then have
\begin{align*}
\Fl^{\ell(X_{1})}_{0,d_{1}}\circ\Fl^{\ell(Z_{1})}_{0,1}\circ\sigma\circ\Fl^{Z_{1}}_{1,0} \circ\Fl^{X_{1}}_{d_{1},0}(a,b) &= \big(a,b,\big(f\circ\Fl^{Z_{1}}_{1,0}\circ\Fl^{X_{1}}_{d_{1},0}\big)(b)\big)
\\
&= \big(a,b,\big(f\circ\Fl^{W}_{1,0}\circ\Fl^{Z_{2}}_{1,0}\circ\Fl^{X_{2}}_{d_{2},0}\big)(b)\big)
\\
&= \big(a,b,\big(f\circ\Fl^{Z_{2}}_{1,0}\circ\Fl^{X_{2}}_{d_{2},0}\big)(b)\big)
\\
&=\Fl^{\ell(X_{2})}_{0,d_{2}}\circ\Fl^{\ell(Z_{2})}_{0,1}\circ\sigma\circ\Fl^{Z_{2}}_{1,0} \circ\Fl^{X_{2}}_{d_{2},0}(a,b)
\\
&= \Fl^{\ell(X_{2})}_{0,d_{2}}\circ\sigma\circ\Fl^{X_{2}}_{d_{2},0}(a,b)
\end{align*}
for $(a,b)$ close to $(0,x)$ in $\RB^{k}\times S_{x}$, giving $T\big((\gamma_{1},[X_{1}]^{\g},d_{1})^{-1}\big) = T\big((\gamma_{2},[X_{2}]^{\g},d_{2})^{-1}\big)$ and hence $T\big(\gamma_{1},[X_{1}]^{\g},d_{1}\big) = T\big(\gamma_{2},[X_{2}]^{\g},d_{2}\big)$.

Now suppose that $\FF$ is regular of codimension $q\leq n$ and that $\big(\gamma_{i},[X_{i}]^{\g},d_{i}\big)$ define the same element in $\HH\big(\pi_{M\times\RB^{n}}^{\g,\FF}\big)$. Then for all $\FF$-invariant sections $\sigma$ defined in a neighbourhood of $y$, we~have
\begin{align*}
\Fl^{\ell(X_{1})}_{0,d_{1}}\circ\Fl^{\ell(Z_{1})}_{0,1}\circ\sigma\circ\Fl^{Z_{1}}_{1,0} \circ\Fl^{X_{1}}_{d_{1},0} &= \Fl^{\ell(X_{1})}_{0,d_{1}}\circ\sigma\circ\Fl^{X_{1}}_{d_{1},0} = \Fl^{\ell(X_{2})}_{0,d_{2}}\circ\sigma\circ\Fl^{X_{2}}_{d_{2},0}
\\
&= \Fl^{\ell(X_{2})}_{0,d_{2}}\circ\Fl^{\ell(Z_{2})}_{0,1}\circ\sigma\circ \Fl^{Z_{2}}_{1,0}\circ\Fl^{X_{2}}_{d_{2},0}
\end{align*}
on some open neighbourhood of $x$ in $M$. Writing $\sigma$ in terms of an $\FF$-invariant function $f$ on $S_{y}$ as in equation~\eqref{sigmaf}, we then have
\begin{gather*}
f\circ\Fl^{Z_{1}}_{1,0}\circ\Fl^{X_{1}}_{d_{1},0}\big|_{S_{x}} = f\circ \Fl^{Z_{2}}_{1,0}\circ\Fl^{X_{2}}_{d_{2},0}\big|_{S_{x}}
\end{gather*}
on some open neighbourhood of $x$ in $S_{x}$. Since slices in a regular foliation carry trivial foliations by points, taking $f$ to be the restriction to $S_{y}$ of any function $U\rightarrow\RB^{q}\hookrightarrow\RB^{n}$ defining $\FF$ on an~open set $U$ containing $y$ we then have
\begin{gather*}
\Fl^{Z_{1}}_{1,0}\circ\Fl^{X_{1}}_{d_{1},0}\big|_{S_{x}} = \Fl^{Z_{2}}_{1,0}\circ\Fl^{X_{2}}_{d_{2},0}\big|_{S_{x}}
\end{gather*}
on some neighbourhood of $x$ in $S_{x}$, hence $\text{Hol}\big(\gamma_{1},[X_{1}]^{\g},d_{1}\big) = \text{Hol}\big(\gamma_{2},[X_{2}]^{\g},d_{2}\big)$.
\end{proof}

The proof of the second part of Theorem~\ref{agree} might seem to suggest that if $\FF$ is given by a~smooth, possibly singular, codimension $q\leq n$ Haefliger structure~\cite{ha3} (i.e., the leaves are locally the level sets of families of first integrals), then the holonomy groupoid $\HH\big(\pi_{M\times\RB^{q}}^{\g,\FF}\big)$ will coincide with the Garmendia--Villatoro groupoid. Somewhat surprisingly this is not true, as can be seen in even the simplest of examples.

\begin{Example}\label{counter1}
The Androulidakis--Skandalis holonomy groupoid (and therefore the Garmen\-dia--Villatoro groupoid) of the foliation $\FF$ of $\RB^{2}$ by concentric circles is the action groupoid $S^{1}\ltimes\RB^{2}$~\cite[Example 2, p.~496]{debord}. On the other hand, since all $\FF$-invariant functions in a~neigh\-bour\-hood of the origin are (germinally) fixed by flows of $\FF$-fields, the holonomy groupoid of the trivial singularly foliated bundle $\RB\times\RB^{2}\rightarrow\RB^{2}$ is, as a set, equal to $S^{1}\times\big(\RB^{2}\setminus\{0_{\RB^{2}}\}\big)\cup\{0_{\RB^{2}}\}$.
\end{Example}

Example~\ref{counter1} motivates the following conjecture, a proof of which would definitively say that conservation laws are \textit{never} sufficient to capture all the holonomy of genuinely singular foliations.

\begin{Conjecture}\label{Conjecture}
Suppose that a foliation $\FF$ of a manifold $M$ is defined by a codimension~$q$ Haefliger structure $H$. Then the holonomy groupoid $\HH\big(\pi_{M\times\RB^{q}}^{\g,\FF}\big)$ is equal to the Garmendia--Villatoro groupoid if and only if $H$ is a regular Haefliger structure.
\end{Conjecture}

It is expected that Conjecture~\ref{Conjecture} will require a study of the relationship between slices and the critical points of first integrals. We~comment on the potential for generalisation of the ideas in this paper to recover the Androulidakis--Skandalis groupoid in full generality in the outlook section at the end of the paper.

\subsection{Functoriality}

In this final section, we show that all of our constructions are functorial for a class of morphisms of singularly foliated bundles over the same base foliation. It is conceivable that these results can be extended to morphisms between singularly foliated bundles over \textit{different} foliations (cf.~\cite[Theorem~6.21]{VilGar1}), however this would likely involve an analogue of the homotopy groupoid of~\cite{VilGar1}, which is beyond the scope of this paper. We~leave this question open to future research.

\begin{Definition}\label{singmorph}
Let $(M,\FF)$ be a singular foliation, let $k$ denote any of the symbols $0,1,\dots$, $\infty,\g$, and let $\pi_{B_{1}}\colon B_{1}\rightarrow M$ and $\pi_{B_{2}}\colon B_{2}\rightarrow M$ be singularly foliated bundles that admit enough conservation laws to order $k$. A~\textit{morphism of order $k$} from $\pi_{B_{1}}$ to $\pi_{B_{2}}$ consists of a surjective bundle morphism $f\colon B_{1}\rightarrow B_{2}$ with the properties that:
\begin{enumerate}\itemsep=0pt
\item[$(1)$] the range of $\ell_{1}$ consists of vector fields which are projectable via $f$ and one has
\begin{gather*}
f_{*}\circ\ell_{1} = \ell_{2},
\end{gather*}
\item[$(2)$] $f$ admits enough conservation laws to order $k$. That is, for every $b\in B_{2}$ there is an open neighbourhood $\OO$ of $b$ in $B_{2}$ and a local section $\kappa\colon \OO\rightarrow B_{2}$ of $f$ for which
\begin{gather*}
\frac{\rm d}{{\rm d}t}\bigg|_{0}\big[\Fl^{\ell_{2}(X)}_{t}\circ\kappa\circ\Fl^{\ell_{1}(X)}_{-t}\big]^{k}_{b} = 0
\end{gather*}
for all $X\in\FF$ defined in a neighbourhood of $\pi_{B_{2}}(b)$.
\end{enumerate}
\end{Definition}

Observe, in the notation of Definition~\ref{singmorph}, that if $\sigma$ is any $\FF$-invariant section of $\pi_{B_{1}}$, then $f\circ\sigma$ is an $\FF$-invariant section of $\pi_{B_{2}}$: for $X\in\FF$ defined in a neighbourhood of $x\in\dom(\sigma)$, we~have
\begin{align*}
\frac{\rm d}{{\rm d}t}\bigg|_{0}\big[\Fl^{\ell_{2}(X)}_{t}\circ (f\circ\sigma)\circ\Fl^{X}_{-t}\big]^{k}_{x} &
= \frac{\rm d}{{\rm d}t}\bigg|_{0}\big[f\circ\Fl^{\ell_{1}(X)}_{t}\circ\sigma\circ\Fl^{X}_{-t}\big]^{k}_{x}
\\
&= {\rm d}[f]^{k}\frac{\rm d}{{\rm d}t}\bigg|_{0}\big[\Fl^{\ell_{1}(X)}_{t}\circ\sigma\circ\Fl^{X}_{-t}\big]^{k}_{x}=0,
\end{align*}
where $[f]^{k}\colon \Gamma_{k}(\pi_{B_{1}})\rightarrow\Gamma_{k}(\pi_{B_{2}})$ is the smooth map $\big(x,[\sigma]^{k}_{x}\big)\mapsto (x,[f\circ\sigma]^{k}_{x})$ induced by $f$ and ${\rm d}[f]^{k}$ its differential (Definition~\ref{differential}). Thus $f$ induces a morphism $f^{\FF}\colon \Gamma_{k}(\pi_{B_{1}})^{\FF}\rightarrow\Gamma_{k}(\pi_{B_{2}})^{\FF}$ of diffeological pseudo-bundles, which is surjective since $f$ admits enough conservation laws to order $k$.

\begin{Theorem}\label{functoriality}
Let $k$ denote any of the symbols $0,1,\dots,\infty,\g$, let $\pi_{B_{1}}\colon B_{1}\rightarrow M$ and $\pi_{B_{2}}\colon$ \mbox{$B_{2}\rightarrow M$} be singularly foliated bundles admitting enough conservation laws to order $k$, and let $f\colon \pi_{B_{1}}\rightarrow\pi_{B_{2}}$ be a morphism of order $k$ in the sense of Definition~$\ref{singmorph}$. Then the identity map $\id\colon \PP(\FF)\rightarrow\PP(\FF)$ descends, for each $l\leq k$, to a morphism $\phi^{l}\colon \HH\big(\pi_{B_{1}}^{l,\FF}\big)\rightarrow\HH\big(\pi_{B_{2}}^{l,\FF}\big)$ of diffeological groupoids for which the diagrams
\begin{equation}
\begin{tikzcd}
\HH\big(\pi_{B_{1}}^{l,\FF}\big) \ar[r,"\phi^{l}"] \ar[d,"\Pi_{B_{1}}^{m,l}"] & \HH\big(\pi_{B_{2}}^{l,\FF}\big) \ar[d,"\Pi_{B_{2}}^{m,l}"]
\\
\HH\big(\pi_{B_{1}}^{m,\FF}\big) \ar[r,"\phi^{m}"] & \HH\big(\pi_{B_{2}}^{m,\FF}\big)
\end{tikzcd}
\end{equation}
commute for all $m\leq l\leq k$. In~particular if $\pi_{B_{1}}$ and $\pi_{B_{2}}$ admit enough conservation laws to order $\g$ and $f$ is a morphism of order $\g$, the diagram
\begin{center}
\begin{tikzcd}[column sep = small]
& & & & \HH\big(\pi_{B_{1}}^{\g,\FF}\big) \ar[dl,"\phi^{\g}"'] \ar[dd] & & &
\\
& & & \HH\big(\pi_{B_{2}}^{\g,\FF}\big) \ar[dd] & & & &
\\
&&& & \HH\big(\pi_{B_{1}}^{\infty,\FF}\big) \ar[dl, "\phi^{\infty}"'] \ar[ddl] \ar[ddr] \ar[ddrrr] &&&
\\
& & & \HH\big(\pi_{B_{2}}^{\infty,\FF}\big) & & & &
\\
& & \cdots\ar[r] & \HH\big(\pi_{B_{1}}^{k+1,\FF}\big) \ar[rr] \ar[dl, "\phi^{k+1}"] & & \HH\big(\pi_{B_{1}}^{k,\FF}\big) \ar[r] \ar[dl,"\phi^{k}"] & \cdots \ar[r] & \HH\big(\pi_{B_{1}}^{0,\FF}\big) \ar[dl,"\phi^{0}"]
\\[2ex]
& \cdots\ar[r] & \HH\big(\pi_{B_{2}}^{k+1,\FF}\big) \ar[from = uur, crossing over] \ar[rr] & & \HH\big(\pi_{B_{2}}^{k,\FF}\big) \ar[from = uul, crossing over] \ar[r] & \cdots \ar[r] & \HH\big(\pi_{B_{2}}^{0,\FF}\big) \ar[from = uulll, crossing over] &
\end{tikzcd}
\end{center}
commutes. Here the unlabelled arrows are as in Theorem~$\ref{hierarchy}$. That is, the hierarchy of holo\-nomy groupoids is functorial.
\end{Theorem}

\begin{proof}
Our first task is to show that the map $\id\colon \PP(\FF)\rightarrow\PP(\FF)$ does indeed descend to a~well-defined morphism of each quotient. Suppose then that $\big(\gamma_{1},[X_{1}]^{\g},d_{1}\big)$ and $\big(\gamma_{2},[X_{2}]^{\g},d_{2}\big)$ in~$\PP(\FF)$ satisfy
\begin{gather}\label{hyp}
T\big(\pi_{B_{1}}^{k,\FF}\big)(\gamma_{1},X_{1},d_{1})\big(x,[\sigma]^{k}_{x}\big) =T\big(\pi_{B_{1}}^{k,\FF}\big)(\gamma_{2},X_{2},d_{2})\big(x,[\sigma]^{k}_{x}\big)
\end{gather}
for all $\big(x,[\sigma]_{x}^{k}\big)$ in $\Gamma_{k}(\pi_{B_{1}})^{\FF}$. Then since $f$ is a morphism, the diagram
\begin{center}
\begin{tikzcd}
\PP(\FF)\times_{s,\pi_{B_{1}}^{k,\FF}}\Gamma_{k}(\pi_{B_{1}})^{\FF} \ar[r, "L_{k}^{1}"] \ar[d, "\id\times f^{\FF}"] & \PP\big(\Gamma_{k}(\pi_{B_{1}})^{\FF}\big) \ar[d,"\PP(f^{\FF})"] \\ \PP(\FF)\times_{s,\pi_{B_{2}}^{k,\FF}}\Gamma_{k}(\pi_{B_{2}})^{\FF} \ar[r, "L_{k}^{2}"] & \PP\big(\Gamma_{k}(\pi_{B_{2}})^{\FF}\big)
\end{tikzcd}
\end{center}
commutes. Therefore, for each $i = 1,2$, we have
\begin{align*}
T\big(\pi_{B_{2}}^{k,\FF}\big)\big(\gamma_{i},[X_{i}]^{\g},d_{i}\big)\big(f^{\FF}\big(x,[\sigma]^{k}_{x}\big)\big)
&= L_{k}^{2}\big(\big(\gamma_{i},[X_{i}]^{\g},d_{i}\big);f^{\FF}\big(x,[\sigma]^{k}_{x}\big)\big)(d_{i})
\\
&= \PP(f^{\FF})\big(L_{k}^{1}\big(\gamma_{i},[X_{i}]^{\g},d_{i};x,[\sigma]^{k}_{x}\big)\big)(d_{i})
\\
&= f^{\FF}\big(T\big(\pi_{B_{1}}^{k,\FF}\big)\big(\gamma_{i},[X_{i}]^{\g},d_{i}\big)\big(x,[\sigma]_{x}^{k}\big)\big)
\end{align*}
for all $\big(x,[\sigma]_{x}^{k}\big)\in\Gamma_{k}(\pi_{B_{1}})^{\FF}$, so surjectivity of $f^{\FF}$ together with the hypothesis \eqref{hyp} tells us that $T\big(\pi_{B_{2}}^{k,\FF}\big)\big(\gamma_{1},[X_{1}]^{\g},d_{1}\big) = T\big(\pi_{B_{2}}^{k,\FF}\big)\big(\gamma_{2},[X_{2}]^{\g},d_{2}\big)$. Therefore $\id\colon \PP(\FF)\rightarrow\PP(\FF)$ does indeed descend to a map $\phi^{k}\colon \HH\big(\pi_{B_{1}}^{k,\FF}\big)\rightarrow\HH\big(\pi_{B_{2}}^{k,\FF}\big)$, whose smoothness follows from that of $\id$, and which is a homomorphism by functoriality of $\id$.

We have thus proved that each of the diagrams\vspace{-1ex}
\begin{center}
\begin{tikzcd}
\PP(\FF) \ar[r,"\id"] \ar[d,"\Pi_{B_{1}}^{k}"] & \PP(\FF) \ar[d,"\Pi_{B_{2}}^{k}"]
\\ \HH\big(\pi_{B_{1}}^{k,\FF}\big) \ar[r,"\phi^{k}"] & \HH\big(\pi_{B_{2}}^{k,\FF}\big)
\end{tikzcd}
\end{center}
commutes, where $\Pi_{B_{i}}^{k}$ is the quotient of $\PP(\FF)$ onto $\HH\big(\pi_{B_{i}}^{k,\FF}\big)$. It follows then that for $l\leq k$, the diagram\vspace{-1ex}
\begin{center}
\begin{tikzcd}
\PP(\FF) \ar[r,"\id"] \ar[d,"\Pi_{B_{1}}^{k}"]& \PP(\FF) \ar[d,"\Pi_{B_{2}}^{k}"] \\ \HH\big(\pi_{B_{1}}^{k,\FF}\big) \ar[r,"\phi^{k}"] \ar[d,"\Pi_{B_{1}}^{l,k}"] & \HH\big(\pi_{B_{2}}^{k,\FF}\big) \ar[d,"\Pi_{B_{2}}^{l,k}"] \\ \HH\big(\pi_{B_{1}}^{l,\FF}\big) \ar[r,"\phi^{l}"] & \HH\big(\pi_{B_{2}}^{l,\FF}\big)
\end{tikzcd}
\end{center}
commutes, and then the result follows by Theorem~\ref{hierarchy}.
\end{proof}

\section{Outlook}\label{sec5}

In the author's estimation, there are three primary questions arising from this work that have yet to be answered.

Firstly, an assumption that we have had to impose in Definition~\ref{singfolbund} is that morphisms of~sheaves of smooth sections are smooth with respect to the diffeology of Proposition~\ref{functprop}. It~is far from clear that this assumption is really necessary. That is, it appears possible that any morphism of sheaves of smooth sections is automatically smooth with respect to this diffeology. Indeed, the domain considerations present in Proposition~\ref{functprop} are automatically satisfied by maps arising from morphisms of sheaves, and attempts thus far to construct a morphism of sheaves which is not smooth with respect to this diffeology have proved unsuccessful. A~proof that any morphism of sheaves of smooth sections is itself diffeologically smooth would allow us to remove these seemingly extraneous assumptions.

Secondly, it is clear from Examples~\ref{triv},~\ref{regex} and~\ref{equiex} that in many situations, a singular partial connection on a fibre bundle induces a singular foliation of its total space by projectable vector fields. This foliation cannot, however, be expected to meet the requirements of Definition~\ref{singdef} (nor the equivalent definitions using compactly supported vector fields, for instance the one used in~\cite{iakovos2}), since projectable vector fields are not closed under multiplication by arbitrary smooth functions on the total space. This suggests that Definition~\ref{singdef} might to be relaxed to allow for closure of vector fields under the sheaf of projectable functions. Such a modification will have no effect on the integration theorem (see~\cite[Theorem~4.2(e)]{sussmann}). Having relaxed Definition~\ref{singdef}, a proof that singularly foliated bundles admit foliations of their total spaces will require a proof that a presheaf of Lie--Rinehart algebras (such as the image of a singular partial connection) becomes a sheaf of Lie--Rinehart algebras under sheafification. This question does not appear to have been studied in the literature.

Finally, and most importantly, it would be interesting to prove Conjecture~\ref{Conjecture} and to~deter\-mine whether the techniques of this paper can be generalised to recover the Androulidakis--Skandalis holonomy groupoid in generality. That is, for any foliation $(M,\FF)$, one seeks a pseudo-bundle $\pi_{\BB}\colon \BB\rightarrow M$ and a lifting map $\PP(\FF)\times_{s,\pi_{\BB}}\BB\rightarrow\PP(\BB)$, for which quotient of $\PP(\FF)$ by the fibres of the associated transport functor is isomorphic to the Androulidakis--Skandalis groupoid. We~have two suggestions in this direction.
\begin{enumerate}\itemsep=0pt
\item It may be possible define a diffeological pseudo-bundle of germs of slices, whose fibre over~$x\in M$ consists of all germs of slices through $x$, using a diffeology similar to that of the pseudo-bundles of germs given in this paper. Since flows of elements of $\PP(\FF)$ send slices to slices, one would obtain the sought-after lifting map using flows. The~results of Garmendia--Villatoro~\cite{VilGar1} suggest that the resulting holonomy groupoid would be isomorphic to the Androulidakis--Skandalis groupoid.
\item As suggested by one of the referees, one could alternatively attempt to equip the space $N\FF:=\sqcup_{x\in M}N_{x}L$ of normal fibres of $(M,\FF)$ considered in~\cite{iakovos4} with a diffeology with respect to which it is a pseudo-bundle over $M$. This pseudo-bundle would admit a (leafwise) Bott connection~\cite[p.~374]{iakovos4}, which could be used to lift elements of $\PP(\FF)$ to $N\FF$.
\end{enumerate}

\subsection*{Acknowledgements}

This research was supported by the Australian Research Council, through the Discovery Project grant DP200100729. I extend special thanks to V.~Mathai, for taking me on as a postdoc at the University of Adelaide, and to I.~Androulidakis, for enlightening discussions and correspondence in late 2018 and for encouraging me to think about singular foliations. I thank B.~McMillan, for~helpful discussions concerning conservation laws. Finally, I extend deep thanks to the anonymous referees, whose careful consideration and critique of the paper have greatly improved its exposition.

\pdfbookmark[1]{References}{ref}
\LastPageEnding

\end{document}